\documentclass[a4paper]{scrartcl}
\usepackage[utf8]{inputenc}
\usepackage{amsmath,amssymb,amsthm,amsfonts}
\usepackage{mathtools,thmtools}
\usepackage[british]{babel}
\usepackage{csquotes}
\usepackage{hyperref}
\usepackage[nameinlink,noabbrev]{cleveref}
\usepackage{enumerate}
\usepackage{xcolor}
\usepackage{comment}
\usepackage{romanbar}

\RequirePackage{tikz}
    \usetikzlibrary{positioning}
    \usetikzlibrary{decorations.pathreplacing}

\usepackage[maxbibnames=5,sorting=nyt,sortcites,giveninits=true]{biblatex}
\DeclareNameAlias{sortname}{family-given}
\usepackage{algorithm, caption}
\newcommand{\R}     {\mathbb{R}}

\newcommand{\eps}   {\varepsilon}

\newcommand{\TV}    {\textup{TV}}

\newcommand{\rn}[1]{\Romanbar{#1}}
\newcommand{\tmix}{t_{\operatorname{mix}}}
\newcommand{\ind}{\mathbf{1}}
\newcommand{\diam}{\operatorname{diam}}
\newcommand{\proj}{\operatorname{proj}}
\newcommand{\diag}{\operatorname{diag}}
\newcommand{\acrj}{\mathrm{a/r}}
\newcommand{\ac}{\mathrm{a}}
\newcommand{\rj}{\mathrm{r}}
\newcommand{\C}{\mathcal{C}}
\newcommand{\lam}{\lambda}
\newcommand{\Lam}{\Lambda}
\newcommand{\cat}{\mathrm{Categorical}}
\newcommand{\NUTS}{\mathrm{NUTS}}
\newcommand{\HMC}{\mathrm{HMC}}

\DeclarePairedDelimiterX{\norm}[1]{\lVert}{\rVert}{#1}

\DeclareMathOperator{\tr}   {tr}

\DeclareMathOperator{\Unif} {Unif}
\DeclareMathOperator{\unif} {unif}

\theoremstyle{plain}
\newtheorem{theorem}{Theorem}
\newtheorem{lemma}[theorem]{Lemma}
\newtheorem{proposition}[theorem]{Proposition}

\theoremstyle{definition}
\newtheorem{definition}[theorem]{Definition}

\newtheorem{remark}[theorem]{Remark}
\newtheorem{remark*}{Remark}

\crefname{lemma}{lemma}{lemmas}
\crefname{theorem}{theorem}{theorems}
\crefname{assumption}{assumption}{assumptions}

\crefname{assualph}{assumption}{assumptions}

\title{On Accelerated Mixing of the\\No-U-turn Sampler}
\author{Stefan Oberd\"orster\thanks{Institute for Applied Mathematics, University of Bonn, 
\href{mailto:oberdoerster@uni-bonn.de}{\texttt{oberdoerster@uni-bonn.de}}}
}
\date{\today}

\begin{document}

\maketitle

\begin{abstract}
Recent progress on the theory of variational hypocoercivity established that Randomized Hamiltonian Monte Carlo---at criticality---can achieve pronounced acceleration in its convergence and hence sampling performance over diffusive dynamics.
Manual critical tuning being unfeasible in practice has motivated automated algorithmic solutions, notably the No-U-turn Sampler.
Beyond its empirical success, a rigorous study of this method's ability to achieve accelerated convergence has been missing.
We initiate this investigation combining a concentration of measure approach to examine the automatic tuning mechanism with a coupling based mixing analysis for Hamiltonian Monte Carlo.
In certain Gaussian target distributions, this yields a precise characterization of the sampler's behavior resulting, in particular, in rigorous mixing guarantees describing the algorithm's ability and limitations in achieving accelerated convergence.
\end{abstract}

\section{Introduction}

Markov chain Monte Carlo (MCMC) concerns itself with the task of sampling complex probability distributions---representations of data.
Akin to the microscope, it reveals empirical insight invisible to the naked eye.
This puts MCMC sampling at the core of natural as well as social sciences \cite{GelmanBook13,andrieu2003introduction,LeRoSt2010A,bardsley2012mcmc,kaipio2005statistical,stuart2010inverse}.

Several factors complicate the sampling from distributions arising in practice.
For instance, realistic models typically involve many parameters, putting the distributions encoding their relationships in spaces of high dimension.
Furthermore, different characteristic scales across parameters or interactions amongst them may produce ill-conditioned and multi-scale distributions combining both directions or regions of concentrated as well as spread out probability mass.
Illustrating toy examples include Gaussian measures comprising both large and small variance directions along with Rosenbrock and funnel distributions \cite{GoodmanWeareEnsemble,NURS}.

In light of these challenges, a particularly notable class of sampling dynamics are non-reversible lifts \cite{EberleLoerler24} which emerge from lifting reversible dynamics to phase space by adding a notion of velocity.
For equilibrium distributions $\mu(dx)\propto e^{-U(x)}dx$ on Euclidean space, a prominent example is Hamiltonian flow
\begin{equation}\label{eq:HFode}
    dX_t\ =\ V_t\,dt\;,\quad dV_t\ =\ -\nabla U(X_t)\,dt
\end{equation}
arising from lifting the overdamped Langevin diffusion
\[ dX_t\ =\ -\nabla U(X_t)\,dt + \sqrt 2\,dB_t\;. \]
Combining Hamiltonian flow with suitable velocity randomization yields ergodic Markov processes.
On one hand, continuous partial randomization via an Ornstein-Uhlenbeck process produces the Langevin diffusion
\begin{equation}\label{eq:LD}
    dX_t\ =\ V_t\,dt\;,\quad dV_t\ =\ -\nabla U(X_t)\,dt - \lambda V_t\,dt + \sqrt{2\lambda}\,dB_t
\end{equation}
with friction $\lambda>0$, see \cite{Pavliotis14}.
On the other, discrete full randomization prompts Randomized Hamiltonian Monte Carlo (Randomized HMC) which follows Hamiltonian flow for integration times $T\sim\mathrm{Exp}(\lambda)$ between which the velocity is fully refreshed from a canonical Gaussian distribution, see \cite{BoSa2017}.

These non-reversible dynamics stand out due to their ability to achieve accelerated convergence to equilibrium compared to the underlying reversible overdamped diffusion.
This was recently shown \cite{lu2022explicit,cao2023explicit,EberleLoerler24} following the variational approach to hypocoercivity developed in \cite{albritton2024variational}.
Specifically, assuming the Poincar\'e inequality
\begin{equation}\label{eq:poincare}
    \int_{\R^d} f^2\,d\mu\ \leq\ \frac1m\int_{\R^d}|\nabla f|^2\,d\mu\quad\text{for $m>0$ and all $f\in H^1(\mu)$ with $\int_{\R^d} f\,d\mu=0$}
\end{equation}
together with a negative lower curvature bound of order $m$, $\nabla^2U\gtrsim -mI_d$, allowing for mild non-convexity, and a superlinear growth condition on $U$ at infinity, the relaxation time of critical Randomized HMC is of order $m^{-1/2}$, see \cite{EberleLoerler24}.\footnote{A similar result holds for the Langevin diffusion \cite{cao2023explicit}. See \cite{EberleLoerler2024manifold} for an extension of the theory to Riemannian manifolds with boundary and \cite{francis2025flowpoincare} for a flexibly applicable framework.}
Criticality refers to the refresh rate $\lambda$ being of order $m^{1/2}$, meaning that the integration times along Hamiltonian flow between full velocity randomizations are of order $m^{-1/2}$.
In contrast, assuming the Poincar\'e inequality \eqref{eq:poincare}, the relaxation time of the overdamped Langevin diffusion is of order $m^{-1}$, see \cite{bakry_book}.
Note that the convergence of neither dynamics explicitly depends on the dimension of the ambient space.
The square root acceleration of critical Randomized HMC compared to overdamped Langevin---known as the diffusive-to-ballistic speed-up---is the best possible acceleration achievable through lifting \cite{EberleLoerler24}.

Heuristically, this can be explained using that the Poincar\'e inequality \eqref{eq:poincare} implies a concentration inequality of the form
\begin{equation}\label{eq:concentration}
    \mu(A_r)\ \geq\ 1-\exp\bigl(-cm^{1/2}r\bigr)\quad\text{for all $r\geq0$ whenever $\mu(A)\geq1/2$,}
\end{equation}
where $c>0$ is some numerical constant and $A_r=\{x\in\R^d:\mathrm{dist}(x,A)<r\}$ denotes the $r$-neighborhood of $A$, see \cite{bakry_book}.
It asserts that virtually all mass of $\mu$ concentrates within a distance of order $m^{-1/2}$.
Traversing this distance ballistically following Hamiltonian flow with unit velocity, which the refreshed velocities of Randomized HMC are in every direction, takes precisely an integration time of order $m^{-1/2}$.
In the assumed absence of relevant barriers, critical Randomized HMC therefore explores the region in which the distribution's mass concentrates globally between consecutive velocity randomizations.
Hence, a constant number of randomizations suffice to converge, corresponding to a relaxation time of order $m^{-1/2}$.
On the other hand, shortening the integration times restricts the motion between randomizations to be increasingly local.
In the overdamped limit, these independent local moves yield diffusive motion which takes physical time of order $m^{-1}$ to explore the distribution's region of concentration.
See \Cref{fig:Gaussian_acceleration} for further illustration in a Gaussian example.
In summary, the diffusive-to-ballistic speed-up in Randomized HMC hinges on critical integration times $T$ of order $m^{-1/2}$.

\begin{figure}
\centering
\begin{tikzpicture}
\draw[dashed] (0,0) ellipse (7 and 1);

\draw[thick] plot[domain=0:16,samples=100] ({cos(\x/7 r)*(-5)+sin(\x/7 r)*7*0.1},{cos(\x r)*(-0.5)+sin(\x r)*0.5});
\draw[gray,thick,->] plot[domain=16:22.2,samples=100] ({cos(\x/7 r)*(-5)+sin(\x/7 r)*7*0.1},{cos(\x r)*(-0.5)+sin(\x r)*0.5});

\draw[line width=1.5pt, ->] (-5,-0.5) -- (-4.9,0) node[left] {$V_0$};
\filldraw (-5,-0.5) circle (2pt) node[right] {$X_0$};
\filldraw ({cos(16/7 r)*(-5)+sin(16/7 r)*7*0.1},{cos(16 r)*(-0.5)+sin(16 r)*0.5}) circle (2pt) node[right] {$X_T$};

\node[gray,above right] at ({cos(22/7 r)*(-5)+sin(22/7 r)*7*0.1+0.1},{cos(22 r)*(-0.5)+sin(22 r)*0.5+0.1}) {U-turn};

\end{tikzpicture}
\caption{\textit{Ballistic motion along Hamiltonian flow \eqref{eq:HFode} in the bivariate Gaussian $\mu=\mathcal N(0,\,\diag(m^{-1},1))$ with $m\ll1$.
At criticality, the integration times $T$ of Randomized HMC are of order $m^{-1/2}$, producing global moves across the ellipse, in which the mass of $\mu$ concentrates, between consecutive velocity randomizations.
Qualitatively shorter integration times restrict to local moves along the large variance direction, slowing convergence.
A U-turn according to \eqref{eq:Uturncondition} occurs after a critical integration time, motivating the No-U-turn Sampler.}}
\label{fig:Gaussian_acceleration}
\end{figure}
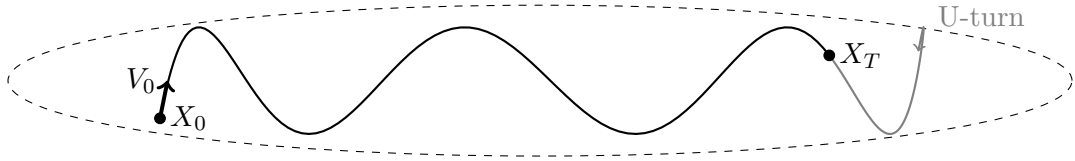

Many particularly challenging target distributions encountered in practical sampling applications feature flat regions or directions.
The resulting small values of $m\ll1$ make diffusive sampling dynamics slow and at times prohibitively computationally expensive.
The diffusive-to-ballistic speed-up of critical Randomized HMC and related dynamics then offers a powerful enhancement of sampling performance.
However, the value and even order of magnitude of $m$ required to realize critical tuning is generally unknown in practice.

Achieving acceleration without knowledge of $m$ inspired Hoffman and Gelman \cite{HoGe2014} to introduce the No-U-turn Sampler (NUTS), which is now the default sampler in many probabilistic programming environments \cite{carpenter2016stan,salvatier2016probabilistic,nimble-article:2017,ge2018t,phan2019composable}.
It is an algorithmic implementation of Randomized HMC viewed as the discrete-time Markov chain on position space obtained from only observing the continuous-time process at its refresh times.
The method's key feature is self-tuning of the integration times or equivalently the velocity refresh rate.
While criticality of Randomized HMC depends on the global quantity $m$, NUTS employs a more local approach: Between velocity randomizations, it builds an orbit along Hamiltonian flow and uses its shape to infer a suitable integration time.
Roughly speaking, given a position $X_0$ (the previous state), the refreshed velocity $V_0$ and a step size $h>0$, NUTS extends the orbit $\{(X_{hi},V_{hi})\}_{i\in I}$ with $0\in I\subset\mathbb Z$ along Hamilton flow \eqref{eq:HFode} until it encounters a U-turn in the sense that
\begin{equation}\label{eq:Uturncondition}
    \min\bigr(v_+\cdot(x_+-x_-),\,v_-\cdot(x_+-x_-)\bigr)\ <\ 0\;,
\end{equation}
where $(x_+,v_+)=(X_{h\max I},V_{h\max I})$ and $(x_-,v_-)=(X_{h\min I},V_{h\min I})$ are the orbit's endpoints in phase space.
From the resulting orbit, the next state of the NUTS chain is selected according to a probability distribution filtering out discretization error introduced by the leapfrog approximation of Hamiltonian flow used in practice.
Importantly, assuming these errors to be controlled, the integration time between consecutive states of NUTS is comparable to the total orbit length in physical time.
Hence, criticality and therewith the diffusive-to-ballistic speed-up in NUTS depend on the orbit length obtained from the U-turn condition \eqref{eq:Uturncondition}.

While NUTS enjoys great popularity among practitioners, theoretical insight into its properties remains scarce.
In the course of a recent uptick in research activity, reversibility \cite{BouRabeeCarpenterMarsden2024,durmus2023convergence}, ergodicity \cite{durmus2023convergence} and non-asymptotic mixing guarantees in canonical Gaussian distributions \cite{BouRabeeOberdoersterNUTS1} has been shown.
However, a central question in light of the method's origin remains open:
\begin{quote}
\emph{Does the No-U-turn Sampler achieve the diffusive-to-ballistic speed-up of critical Randomized HMC?}
\end{quote}

This work makes a first step towards this question by analyzing NUTS in Gaussian target distributions.
While this setting might seem insignificant at first sight, it requires a surprisingly rich mathematical theory.
In particular, it allows for a very precise rigorous description of the method's behavior, resulting in sharp mixing guarantees capable of capturing the diffusive-to-ballistic speed-up.
Beyond Gaussian distributions, the theoretical foundation requires substantial development in order to attack this question.

Concretely, we prove a concentration inequality for the U-turn property, from which we obtain an accurate characterization of the U-turn based orbit selection.
This enables us to precisely predict the orbit length selected by NUTS in certain settings and therefore positions us to investigate the question regarding accelerated convergence posed above.

Specifically, we first consider isotropic Gaussian distributions for which we find NUTS to select orbits of critical length under natural assumptions, indicating accelerated convergence.

Subsequently, we study the more challenging class of two-scale Gaussian distributions, products of two isotropic Gaussians.
While, within this group of targets, we discover a phase of distributions in which NUTS reliably selects orbits sufficiently long for acceleration, we however also show the existence of two-scale Gaussians in which NUTS with certain fixed step sizes is limited to short orbits not sufficient for acceleration.
In a suitable infinite dimensional limit, we further show this dichotomy to strengthen into an abrupt phase transition separating two-scale Gaussians in which NUTS selects long orbits for almost all step sizes from ones in which NUTS selects short orbits for a set of step sizes of strictly positive measure.

This insight into the orbit selection of NUTS is subsequently combined with a state-of-the-art mixing analysis via couplings.
In particular, we show that if NUTS selects orbits of critical length, it achieves the diffusive-to-ballistic speed-up.

The paper is structured as follows:
In the next section, we formally introduce NUTS.
In \textsection\ref{sec:conc}, we show concentration of the U-turn property, its implications on orbit selection, and subsequently thoroughly discuss the results in isotropic and two-scale Gaussian distributions.
Finally, \textsection\ref{sec:mixing} is devoted to mixing guarantees for NUTS in two-scale Gaussians.
The main results are primarily discussed around \Cref{prop:2S} and \Cref{thm:NUTSmixing}.

\subsection*{Acknowledgements}

I thank Nawaf Bou-Rabee and Andreas Eberle for their guidance and support.
I further thank Bob Carpenter, Yuansi Chen and Francis L\"orler for valuable discussions.
Financial supported by the Deutsche Forschungsgemeinschaft (DFG, German Research Foundation) under Germany’s Excellence Strategy – EXC-2047/1 – 390685813 is gratefully acknowledged.

\section{The No-U-turn Sampler}\label{sec:NUTS}

In this section, we formally introduce the No-U-turn Sampler (NUTS).
After fixing some basic notation, we first discuss Hamiltonian Monte Carlo and the U-turn property, the fusion of which yields NUTS.
Subsequently, we describe NUTS in detail.
Our presentation follows \cite{BouRabeeOberdoersterNUTS1}.
Beyond the classical algorithm, which uses the leapfrog integrator to discretize Hamiltonian flow, we define a variant employing exact Hamiltonian flow.
This allows for a simplified analysis of the U-turn mechanism---the core feature of NUTS.

We focus on target probability distributions $\mu$ on Euclidean space $\mathbb R^d$ with continuously differentiable density also denoted by $\mu$.  The Hamiltonian
\begin{equation*}\label{eq:H}
    H(x,v)\ =\ -\log\mu(x)\ +\ \frac12|v|^2\;,
\end{equation*}
where $x \in \mathbb{R}^d$ and $v \in \mathbb{R}^d$ represent position and velocity, respectively, extends the target to the Boltzmann distribution $\mu\otimes\gamma_d$ on phase space $\mathbb R^{2d}$, where $\gamma_d$ denotes the canonical Gaussian measure on $\mathbb R^d$.
Let $\phi_t:\mathbb{R}^{2d}\to\mathbb{R}^{2d}$ with $t\in\mathbb R$ be Hamiltonian flow solving
\[ \frac{d}{dt}\,\phi_t(x,v)\ =\ \bigr(v,\nabla\log\mu(x)\bigr)\quad\text{with $\phi_0=\mathrm{id}$.} \]
Further, let $\Phi_h: \mathbb{R}^{2d} \to \mathbb{R}^{2d}$ denote one corresponding leapfrog step of size $h>0$.  For any $i \in \mathbb{Z}$, define $\Phi_{h}^i:  \mathbb{R}^{2d} \to \mathbb{R}^{2d}$ recursively by $\Phi^0_{h}=\mathrm{id}$ and $\Phi^{i + 1}_{h} = \Phi_{h} \circ \Phi^i_{h}$.
For convenience, we refer to $\phi_t$ with $t\in\mathbb R$ as Hamiltonian flow and to $\Phi_h^{t/h}$ with $t\in h\mathbb Z$ as leapfrog flow.
As we work with both, we write $\varphi_t$ as placeholder for either.

Define the projections from phase space to position and velocity space $\proj_1^d,\proj_2^d:\mathbb R^{2d}\to\mathbb R^d$ by
\begin{equation}\label{eq:proj}
    \bigr(\proj_1^d(x,v),\,\proj_2^d(x,v)\bigr)\ =\ (x,v)\quad\text{for all $(x,v)\in\mathbb R^{2d}$.}
\end{equation}

\subsection{Hamiltonian Monte Carlo Methods}\label{sec:HMC}

We now give a general definition of the family of Hamiltonian Monte Carlo methods with full velocity randomization.

\begin{definition}\label{def:HMC}
Let $\tau:\mathbb R^{2d}\to\mathcal P(\mathbb R)$ map $(x,v)$ to a probability measure $\tau_{x,v}$ on $\mathbb R$.
A Markov transition kernel $\pi_{\HMC(\tau)}$ belongs to the family of Hamiltonian Monte Carlo (HMC) methods with full velocity randomization if its transitions $X \sim \pi_{\HMC(\tau)}(x,\cdot)$ for $x \in \mathbb{R}^d$ take the form
\begin{equation}\label{eq:HMCkernel}
    X\ =\ \proj_1^d\varphi_T(x,v) \quad\text{with $v\sim\gamma_d$ and $T\sim\tau_{x,v}$,}
\end{equation}
where we additionally require $\mathrm{supp}(\tau_{x,v})=h\mathbb Z$  for all $(x,v)\in\mathbb R^{2d}$ in case $\varphi_t=\Phi_h^{t/h}$ represents leapfrog flow.
\end{definition}

In words, the transitions of an HMC method with full velocity randomization proceed as follows: Given the initial position $x$, a velocity $v$ is drawn from the canonical Gaussian measure. From the resulting point in phase space, Hamiltonian or leapfrog flow is run for an integration time $T\sim\tau_{x,v}$. The obtained position is then selected as the next state of the chain.

This definition covers a wide variety of Monte Carlo methods based on Hamiltonian flow.
For instance, deterministic integration times $\tau\equiv\delta_t$ for fixed $t$ yield classical (unadjusted) HMC.
Restricting to one leapfrog step, $t=h$, induces one transition of the unadjusted Langevin algorithm, the Euler discretization of the overdamped Langevin diffusion.
Adding a Metropolis filter via
\[ \tau_{x,v}\ =\ e^{-(H\circ\Phi_h^{t/h}-H)^+(x,v)}\delta_t\ +\ \bigr(1-e^{-(H\circ\Phi_h^{t/h}-H)^+(x,v)}\bigr)\delta_0 \]
produces Metropolis-adjusted HMC for $t\in h\mathbb Z$ and the Metropolis-adjusted Langevin algorithm (MALA) for $t=h$.

Of particular interest to us are HMC methods whose integration times are randomized beyond the elimination of discretization error.
Randomized HMC, as discrete-time Markov chain on position space, arises from the definition with $\tau\equiv\mathrm{Exp}(\lambda)$ for some $\lambda>0$.
NUTS also falls within the definition's scope, see \textsection\ref{sec:integration time}.

HMC methods with partial velocity randomization such as discretizations of the Langevin diffusion \eqref{eq:LD} are not covered by the definition.

\subsection{The U-turn Property}\label{sec:u-turn}

A fundamental question in the design of Hamiltonian Monte Carlo methods with full velocity randomization is: Given a state $x$ and a velocity $v$, what integration time distributions $\tau_{x,v}$ corresponds to criticality, achieving the diffusive-to-ballistic speed-up?
As discussed in the introduction, knowledge of global information about the target distribution's geometry, as for instance encoded in the Poincar\'e inequality \eqref{eq:poincare}, yields a uniform answer.
In applications, however, such global information is typically unavailable.
In addition, target distributions may feature local geometries that vary throughout space.
Then, the optimal integration time distributions promoting the shortest possible integration times while achieving accelerated convergence vary as well.

The No-U-turn architecture was introduced to substitute explicit knowledge of the target's geometry with a subtle adaptation strategy: It constructs orbits along Hamiltonian flow and infers the integration time from the geometric information encoded in their shape via the U-turn property.
Specifically, define an index orbit as a set of consecutive integers $I\subset\mathbb Z$ which, together with a set of initial conditions $(x,v)\in\mathbb R^{2d}$, defines an orbit $\{\varphi_{hi}(x,v)\}_{i\in I}$, a collection of states in phase space following either Hamiltonian or leapfrog flow.
Correspondingly, we refer to such orbits as Hamiltonian orbits and leapfrog orbits.
An orbit has the U-turn property if and only if
\begin{equation}\label{eq:u-turn}
    \min\bigr(v_+\cdot(x_+-x_-),\,v_-\cdot(x_+-x_-)\bigr)\ <\ 0\;,
\end{equation}
where $(x_+,v_+)=\varphi_{h\max I}(x,v)$ and $(x_-,v_-)=\varphi_{h\min I}(x,v)$ are the orbit's endpoints.
As, given $(x,v)$, index orbits and orbits are dual, we equivalently say that an orbit and the corresponding index orbit have the U-turn property.

\subsection{The Sub-U-turn Property and Orbit Construction}

While Randomized HMC as continuous-time Markov process in phase space is non-reversible, which is crucial to its ability of achieving accelerated convergence, as discrete-time Markov chain in position space it is reversible.
In order to ensure NUTS approximates the correct distribution, its transitions are also designed to be reversible with respect to the target distribution \cite{BouRabeeCarpenterMarsden2024,durmus2023convergence}.

Constructing orbits while checking for U-turns in a reversible way leads to a quite involved architecture.
Starting from the initial index orbit $I=\{0\}$ corresponding to the trivial orbit only containing the initial point in phase space, the orbit construction iteratively doubles the index orbit in either direction and checks for U-turns at each iteration.
In particular, reversibility requires not only to check the extended orbit for a U-turn but also the extension and all orbits obtainable from it by repeated halving \cite{BouRabeeCarpenterMarsden2024}.
This yields the following sub-U-turn property.

Given an index orbit $I$ of length $|I|$ being a power of $2$, define $\mathfrak I(I)$ as the collection of index orbits obtainable from $I$ by repeated halving, i.e.,
\[
\mathfrak{I}(I)\ =\ \bigr\{ I_{j,m}\ :\ j \in \{ 0, \dots, \log_2|I| \}, \, m \in \{ 1, \dots, 2^j\}  \bigr\}
\]
where, for all $j\in\{0,\dots,\log_2|I|\}$,  \( I_{j,1},\dots, I_{j,2^j} \) are the unique index orbits of length $|I|2^{-j}$ such that
\[ I\ =\ I_{j,1}\ \cup\ \cdots\ \cup\ I_{j,2^j}\;. \]
An index orbit $I$ is said to have the sub-U-turn property if any index orbit in $\mathfrak I(I)$ has the U-turn property.

We can now describe the orbit construction of NUTS in detail:
Given $(x,v)\in\mathbb R^{2d}$ and a maximal index orbit length of $2^{k_{\max}}$ where $k_{\max} \in \mathbb{N}$, NUTS constructs an orbit $\{\varphi_{hi}(x,v)\}_{i\in I}$ by proceeding iteratively as follows, starting with $I_0=\{0\}$: 
\begin{itemize}
\item For the current orbit $I_j$, draw an extension $I'$ uniformly from $\{I_j-|I_j|,\,I_j+|I_j|\}$.  
\item If  $I'$ has the sub-U-turn property, stop the procedure and select $I_j$ as final orbit.  
\item If $I'$ does not have the sub-U-turn property, extend the orbit by setting $I_{j+1}=I_j\cup I'$.  If $I_{j+1}$ satisfies the U-turn property or  $|I_{j+1}|=2^{k_{\max}}$, stop the procedure and select $I_{j+1}$ as final orbit. 
\item Otherwise, iterate with $I_{j+1}$ as the current orbit. 
\end{itemize}
This generates a sample $I$ from a probability distribution $\mathcal{O}_{x,v}$ over the collection of index orbits.
We say NUTS selects an orbit from this distribution.

\subsection{The Transitions of the No-U-turn Sampler}

Given the current state $x\in\mathbb R^d$, a transition of NUTS proceeds as follows:
Similar to other HMC methods with full velocity randomization, see \textsection\ref{sec:HMC}, NUTS first draws an initial velocity $v\sim\gamma_d$.
The integration time $T\sim\tau_{x,v}$ then results from a two step procedure.
First, an orbit is selected via the U-turn based orbit construction described above.
Second, the next state of NUTS is selected from the positions within the orbit according to a Boltzmann-weighted categorical distribution.
The categorical distribution generalizes the Bernoulli distribution to arbitrary index sets.
Specifically, $\iota \sim\cat(a_i)_{i\in I}$ with summable weights $(a_i)_{i\in I}$ iff
\[ \mathbb P(\iota=i)\ \propto\ a_i \;. \]
A complete transition of NUTS is given in Algorithm~\ref{algo:NUTS}.
The state selection from the orbit is equivalently written as an index selection from the corresponding index orbit.
The method has two user-tuned hyperparameters: the step-size $h>0$ and the maximum number of orbit doublings $k_{\max} \in \mathbb{N}$ which limits the index orbit length to $2^{k_{\max}}$.

\begin{algorithm}[ht] 
\caption{The No-U-turn Sampler $\ X\sim\pi_{\NUTS}(x,\cdot)$}\label{algo:NUTS}
1. Velocity randomization:  $v\sim\gamma_d$.\\
2. Orbit selection:  $I\sim\mathcal{O}_{x,v}$.\\
3. Index selection:  $\iota\sim\cat\bigr(e^{-(H\circ\varphi_{hi}-H)(x,v)}\bigr)_{i\in I}$. \\ 
4. Output: $X=\proj_1^d\varphi_{h\iota}(x,v)$.
\end{algorithm}

Note that this definition comprises two variants of NUTS---one using leapfrog flow and one using Hamiltonian flow.
The former is the implementable method used in practice while the latter is introduced to simplify the theoretical study of NUTS.
The main simplification occurs in the Boltzmann-weighted index selection, which removes leapfrog discretization error, comparable to a Metropolis filter.
As Hamiltonian flow preserves the Hamiltonian, this selection then becomes uniform.

Observe that NUTS using Hamiltonian flow constructs orbits iterating through the same index orbits $I\subset\mathbb Z$ as NUTS using leapfrog flow with step-size $h>0$.
Both variants therefore share the same discretization of physical integration time to $h\mathbb Z$.
In particular, $h$ persists as hyper-parameter in NUTS using Hamiltonian flow.
This ensures the transitions to be reversible.
A variant of NUTS that constructs orbits in a more continuous way has not yet been described.

\subsection{Integration Time in the No-U-turn Sampler}\label{sec:integration time}

With integration time distribution $\mathrm{Law}(h\iota)$, the No-U-turn Sampler fits our definition of an HMC method with full velocity randomization.
For NUTS using Hamiltonian flow, $\iota\sim\Unif(I)$ so that the integration times are comparable to the physical time orbit length $h(|I|-1)$.
We therefore call an orbit length critical if it is of the same order as the critical integration times.
For NUTS using leapfrog flow, the integration time $T$ additionally depends on the discretization errors $H\circ\Phi_h^i-H$ along the orbit.
In particular, excessively growing errors yield decaying categorical weights making parts of the orbit unavailable to index selection.
While this wastes computational resources, more importantly, it could hinder accelerated convergence even if NUTS selects orbits of critical length.
It is hence crucial to control energy errors along the selected orbits by either using a sufficiently small step size $h$ or by locally adapting it, see \cite{BouRabeeCarpenterMarsden2024,BouRabeeCarpenterKleppeMarsden2024,bourabee2025withinorbitadaptiveleapfrognouturn}.
Then, in both variants of NUTS, the integration time is comparable to the physical time orbit length and criticality of integration times corresponds to criticality of physical time orbit lengths.

\section{Concentration of the U-turn Property and Orbit Selection}

As discussed above, accelerated convergence of Randomized HMC hinges on critical integration times, which in NUTS correspond to critical orbit lengths.
Therefore, to describe acceleration in NUTS, it is crucial to understand the U-turn based orbit selection.

In this section, we establish a concentration inequality for the U-turn property in arbitrary Gaussian distributions and derive from it a precise characterization of the orbit selection, see \textsection\ref{sec:conc} and \textsection\ref{sec:orbit}.
Subsequently, we discuss two important examples in detail: Isotropic Gaussians in \textsection\ref{sec:iso} and two-scale Gaussians in \textsection\ref{sec:2S}.
Under natural assumptions, in the former, we discover NUTS to select critical orbits.
In the latter, we reveal two phases of two-scale Gaussians: One in which NUTS selects orbits sufficiently long for acceleration for all step sizes, while in the other being limited to short orbits not sufficient for acceleration for certain step sizes.
Finally, in \textsection\ref{sec:harmonicchain}, we briefly discuss an example illustrating the limitations of the concentration approach.

\subsection{Concentration of the U-turn Property}\label{sec:conc}

The following definition describes the U-turn diagnostic determining the U-turn property.

\begin{definition}\label{def:f}
For $h>0$, $t_-,t_+\in h\mathbb Z$, and $(x,v)\in\mathbb R^{2d}$ define
\begin{equation}\label{eq:f_concentration}
    f(x,v,t_-,t_+)\ =\ \min\bigr(v_+\cdot(x_+-x_-),\,v_-\cdot(x_+-x_-)\bigr)
\end{equation}
where $(x_-,v_-)=\varphi_{t_-}(x,v)$ and $(x_+,v_+)=\varphi_{t_+}(x,v)$.
In case of Hamiltonian flow $\varphi=\phi$, the definiition extends to all $t_-,t_+\in\mathbb R$.
\end{definition}

As defined in \textsection\ref{sec:u-turn}, an orbit $\{\varphi_{hi}(x,v)\}_{i\in I}$ has the U-turn property if and only if
\begin{equation*}
  f(x,\,v,\,h\min I,\,h\max I)  \ <\ 0\;.
\end{equation*}

Denote by $\gamma^\C$ the centered Gaussian measure on $\mathbb R^d$ with covariance matrix $\C$.
Without loss of generality, we can assume $\C$ to be diagonal.
In the following, we define scalar functions of diagonal matrices to be the diagonal matrices obtained by applying the function to each diagonal entry, see \eqref{eq:scalarondiag}.
The following theorem states that the U-turn property for Hamiltonian orbits concentrates in the sense that the diagnostic $f$ satisfies a Bernstein-type concentration inequality.
An analogous statement holds for leapfrog orbits.

\begin{theorem}\label{thm:U-turn_concentration}
Let $\mu=\gamma^\C$ and $(x,v)\sim\mu\otimes\gamma_d$.
Further, let $t_-,t_+\in h\mathbb Z$ and $f$ be as defined in \eqref{eq:f_concentration} using Hamiltonian flow.
Then, there exists an absolute constant $c>0$ such that for all $r\geq0$,
\[\begin{aligned}[t]
    &\mathbb P\left(\Bigr|f(x,v,t_-,t_+)-\tr\bigr(\sin\bigr(\C^{-1/2}(t_+-t_-)\bigr)\C^{1/2}\bigr)\Bigr|\geq r\right) \\
    &\qquad\leq\ 4\,\exp\left(-c\,\min\left(\frac{r^2}{\tr(\C)},\,\frac{r}{\|\C^{1/2}\|}\right)\right)\;. 
\end{aligned}\]
\end{theorem}

As $\|\C^{1/2}\|\leq\tr(\C)^{1/2}$, the theorem asserts that for all $t_-,t_+\in h\mathbb Z$,
\begin{equation}\label{eq:f_deviation}
    f(x,v,t_-,t_+)\ =\ \tr\bigr(\sin\bigr(\C^{-1/2}(t_+-t_-)\bigr)\C^{1/2}\bigr)\ +\ O\bigr(\tr(\C)^{1/2}\bigr)
\end{equation}
with high probability for $(x,v)\sim\mu\otimes\gamma_d$, i.e., for virtually all typical points with respect to the equilibrium distribution in phase space.
Note that the first term on the right hand side, which we denote by $f_{\unif}$ and refer to as the uniform term, only depends on the physical time length $t_+-t_-=h(|I|-1)$ of the orbit $\{\phi_{hi}(x,v)\}_{i\in I}$, while local effects in $x,v$ are confined to the deviation term at most of order $\tr(\C)^{1/2}$.
If the uniform term dominates the deviation term, the U-turn property is predominately dictated by the physical time orbit length.

Let us prove the theorem.
The argument identifies the terms in $f$ as bilinear forms and subsequently employs the Hanson-Wright inequality, a Bernstein-type concentration inequality for bilinear forms, see \cite{Vershynin} for details.

\begin{proof}[Proof of \Cref{thm:U-turn_concentration}]
Suppose $(x,v)\sim\gamma^\C\otimes\gamma_d$ and let $t_-,t_+\in h\mathbb Z$, $(x_+,v_+)=\phi_{t_+}(x,v)$ and $(x_-,v_-)=\phi_{t_-}(x,v)$.
We show
\begin{align}\label{eq:firstdisplay}
    &\mathbb P\left(\Bigr|v_+\cdot(x_+-x_-)-\tr\bigr(\sin\bigr(\Lam^{1/2}(t_+-t_-)\bigr)\Lam^{-1/2}\bigr)\Bigr|\geq r\right)\nonumber\\
    &\qquad\leq\ 2\,\exp\left(-c\,\min\left(\frac{r^2}{\tr(\C)},\,\frac{r}{\|\C^{1/2}\|}\right)\right)
\end{align}
for all $r\geq0$ and an absolute constant $c>0$.
Together with an analogous bound for the second scalar product in $f$, we obtain the asserted inequality as follows:
\begin{align*}
&\mathbb P\left(\Bigr|\min\bigr(v_+\cdot(x_+-x_-),\,v_-\cdot(x_+-x_-)\bigr)-\tr\bigr(\sin\bigr(\C^{-1/2}(t_+-t_-)\bigr)\C^{1/2}\bigr)\Bigr|\geq r\right)\\
&\qquad\leq\ \mathbb P\left(\Bigr|v_+\cdot(x_+-x_-)-\tr\bigr(\sin\bigr(\C^{-1/2}(t_+-t_-)\bigr)\C^{1/2}\bigr)\Bigr|\geq r\right)\\
&\qquad\qquad +\ \mathbb P\left(\Bigr|v_-\cdot(x_+-x_-)-\tr\bigr(\sin\bigr(\C^{-1/2}(t_+-t_-)\bigr)\C^{1/2}\bigr)\Bigr|\geq r\right)\\
&\qquad\leq\ 4\,\exp\left(-c\,\min\left(\frac{r^2}{\tr(\C)},\,\frac{r}{\|\C^{1/2}\|}\right)\right)\;.
\end{align*}
Let $\Lam=\C^{-1}$ and assume
\[ \Lam\ =\ \diag(\lam_i)_{1\leq i\leq d}\quad\text{with $\lam_i>0$.} \]
Hamiltonian flow corresponding to the energy $H(x,v)=\frac12|\Lam^{1/2}x|^2+\frac12|v|^2$ takes the explicit form
\begin{equation}\label{eq:HF}
    \phi_t(x,v) =\, \Bigr(\cos(\Lam^{1/2}t)\,x + \sin(\Lam^{1/2}t)\Lam^{-1/2}\,v,\ -\sin(\Lam^{1/2}t)\Lam^{1/2}\,x + \cos(\Lam^{1/2}t)\,v\Bigr)
\end{equation}
where
\begin{equation}\label{eq:scalarondiag}
    \cos(\Lam^{1/2}t)\ =\ \diag\bigr(\cos(\lam_j^{1/2}t)\bigr)_{1\leq j\leq d}
\end{equation}
and $\sin(\Lam^{1/2}t)$ is defined analogously.
The first scalar product in the U-turn property can then be expressed as
\begin{equation}\label{eq:Uturn_quad_form}
    v_+\cdot(x_+-x_-)\ =\ \begin{pmatrix} \Lam^{1/2}x \\ v \end{pmatrix} \cdot A \begin{pmatrix} \Lam^{1/2}x \\ v \end{pmatrix}\quad\text{with}\quad A\ =\ \begin{pmatrix} A^{xx} & A^{xv} \\ A^{xv} & A^{vv} \end{pmatrix}
\end{equation}
where
\begin{align*}
A^{xx}\ &=\ -\sin(\Lam^{1/2}t_+)\bigr(\cos(\Lam^{1/2}t_+)-\cos(\Lam^{1/2}t_-)\bigr)\Lam^{-1/2}\;, \\
A^{xv}\ &=\  \frac12\Bigr(
\begin{aligned}[t]
&\cos(\Lam^{1/2}t_+)\bigr(\cos(\Lam^{1/2}t_+)-\cos(\Lam^{1/2}t_-)\bigr)\\
&-\sin(\Lam^{1/2}t_+)\bigr(\sin(\Lam^{1/2}t_+)-\sin(\Lam^{1/2}t_-)\bigr)\Bigr)\Lam^{-1/2}\;,
\end{aligned}\\
A^{vv}\ &=\ \cos(\Lam^{1/2}t_+)\bigr(\sin(\Lam^{1/2}t_+)-\sin(\Lam^{1/2}t_-)\bigr)\Lam^{-1/2}\;.
\end{align*}
Note that if $(x,v)\sim\gamma^\C\otimes\gamma_d$, then $(\Lam^{1/2}x,v)\sim\gamma_{2d}$.
The Hanson-Wright inequality \cite{Vershynin} asserts
\begin{equation}\label{eq:HW}
\mathbb P\Bigr(\Bigr|v_+\cdot(x_+-x_-)-\mathbb E\bigr(v_+\cdot(x_+-x_-)\bigr)\Bigr|\geq r\Bigr)\ \leq\ 2\,\exp\Bigr(-c'\min\bigr(r^2\|A\|_F^{-2},\,r\|A\|^{-1}\bigr)\Bigr)
\end{equation}
for all $r\geq0$ and an absolute constant $c'>0$.
As
\begin{equation}\label{eq:Uturnmean}
    \mathbb E\bigr(v_+\cdot(x_+-x_-)\bigr)\ =\ \tr\bigr(\sin\bigr(\Lam^{1/2}(t_+-t_-)\bigr)\Lam^{-1/2}\bigr)\ =\ \tr\bigr(\sin\bigr(\C^{-1/2}(t_+-t_-)\bigr)\C^{1/2}\bigr)\;,
\end{equation}
Equation \eqref{eq:firstdisplay} results from bounding $\|A\|_F^2$ and $\|A\|$.
Note that $A$ is a permutation of
\[ \begin{pmatrix} M_1 & & \\  & \ddots & \\ & & M_d\end{pmatrix}\quad\text{where}\quad M_j\ =\ \begin{pmatrix} A^{xx}_{jj} & A^{xv}_{jj} \\ A^{xv}_{jj} & A^{vv}_{jj} \end{pmatrix}\;. \]
Computing the eigenvalues of $M_j$ yields
\begin{align*}
&\|M_j\|\ =\ \frac12\max\nolimits_\pm\Bigr|\sin\bigr(\lam_j^{1/2}(t_+-t_-)\bigr)\pm2\sin\Bigr(\frac12\lam_j^{1/2}(t_+-t_-)\Bigr)\Bigr|\lam_j^{-1/2}\ \leq\ \frac{3\sqrt 3}{4}\lam_j^{-1/2}\;, \\
&\|M_j\|_F^2\ =\ \sin^2\Bigr(\frac12\lam_j^{1/2}(t_+-t_-)\Bigr)\Bigr(\cos\bigr(\lam_j^{1/2}(t_+-t_-)\bigr)+3\Bigr)\lam_j^{-1}\ \leq\ 2\lam_j^{-1}\;.
\end{align*}
Therefore,
\begin{align*}
&\|A\|\ =\ \max\nolimits_j\|M_j\|\ \leq \frac{3\sqrt3}4\max\nolimits_j\lam_j^{-1/2}\ =\ \frac{3\sqrt3}4\|\Lam^{-1/2}\|\ =\ \frac{3\sqrt3}4\|\C^{1/2}\| \quad\text{and}\\
&\|A\|_F^2\ =\ \sum_j\|M_j\|_F^2\ \leq\ 2\sum_j\lam_j^{-1}\ =\ 2\tr(\Lam^{-1})\ =\ 2\tr(\C)\;,
\end{align*}
where we used that $\Lam^{-1}=\C$.
Inserting this together with \eqref{eq:Uturnmean} into \eqref{eq:HW} shows \eqref{eq:firstdisplay}, finishing the proof.
\end{proof}

\subsection{Uniformity of the Orbit Selection}\label{sec:orbit}

The next proposition captures the implications of U-turn property concentration on the orbit selection in NUTS.

For $h>0$ and $k_{\max}\in\mathbb N$, denote the set of physical time orbit lengths that may appear in the orbit selection by
\begin{equation}\label{eq:frakT}
    \mathfrak T\ =\ \bigr\{h(2^k-1)\ :\ k\in\mathbb N,\,k\leq k_{\max}\bigr\}\;.
\end{equation}
Let $\mathfrak I_{\max}$ denote the collection of all index orbits that may be checked for U-turns in the orbit selection.
For $t\in\mathfrak T$, let $\mathfrak I_0(t)$ denote the collection of index orbits $I$ such that $h(|I|-1)=t$ and $0\in I$.
In other words, the index orbits in $\mathfrak I_0(t)$ correspond to orbits of physical time length $t$ that include the initial point $(x,v)$.

\begin{proposition}\label{prop:unif_orbit}
Let $h>0$, $k_{\max}\in\mathbb N$, and $\mathfrak T$ as in \eqref{eq:frakT}.
Let $(x,v)\in\mathbb R^{2d}$ and $f$ as defined in \eqref{eq:f_concentration}.
Assume there exists $\delta>0$ and $f_{\unif}:[0,\infty)\to\mathbb R$ with $f_{\unif}(0)=0$ such that for all $t_-,t_+\in h\mathbb Z$ for which $t_-/h$ and $t_+/h$ correspond to the endpoints of an index orbit in $\mathfrak I_{\max}$, it holds that
\begin{equation}\label{eq:prop_deviation}
    \bigr|f(x,v,t_-,t_+)-f_{\unif}(t_+-t_-)\bigr|\ \leq\ \delta\;.
\end{equation}
Further suppose
\begin{equation}\label{eq:prop_T}
    \mathfrak T\ \cap\ \{-\delta\leq f_{\unif}<\delta\}\ =\ \emptyset\;.
\end{equation}
Then, the orbit selection of NUTS both using leapfrog and Hamiltonian flow simplifies to
\[ \mathcal{O}_{x,v}\ =\ \Unif(\mathfrak I_0(t_\ast))\quad\text{where}\quad t_\ast\ =\ \inf\bigr\{t\in\mathfrak T\ :\ f_{\unif}(t)<0\bigr\}\ \wedge\ \max\mathfrak T\;. \]
\end{proposition}

In case of pronounced concentration of the U-turn property, \eqref{eq:prop_deviation} holds with high probability for a small $\delta$ relative to the scale of $f_{\unif}$.
In the examples discussed below, this limits the size of $\{-\delta\leq f_{\unif}<\delta\}$ and therewith the restriction on the step size $h$ imposed in \eqref{eq:prop_T}.

According to the proposition, the orbit selection in NUTS then simplifies to a uniform selection from the index orbits of fixed physical time length $t_\ast$ containing $0$.
Importantly, $t_\ast$ is deterministic, only depending on the target's covariance $\C$ through $f_{\unif}$ and the user-tuned parameters $h$ and $k_{\max}$ through $\mathfrak T$.
This allows us to identify the orbit length selected by NUTS and in particular under what conditions NUTS selects orbits of critical length.

Furthermore, the asserted uniformity in $x$ and $v$ of the orbit selection will play a crucial role in the mixing analysis of NUTS via coupling techniques, see \textsection\ref{sec:mixing}.

\begin{proof}[Proof of \Cref{prop:unif_orbit}]
The argument is illustrated in \Cref{fig:sinm}.
Given an initial point $(x,v)\in\mathbb R^{2d}$ in phase space, \eqref{eq:prop_deviation} confines $f(x,v,t_-,t_+)$ to a $\delta$-neighborhood around $f_{\unif}(t_+-t_-)$ for all $t_-,t_+$ corresponding to index orbits of orbits that may be checked for U-turns in the orbit selection.
In particular, for $t_+-t_-$ outside of $\{-\delta\leq f_{\unif}<\delta\}$, the sign of $f$ is guaranteed to coincide with the sign of $f_{\unif}$.
Assumption \eqref{eq:prop_T}, which prohibits the physical time orbit lengths $\mathfrak T$ that may be checked during orbit selection to fall within $\{-\delta\leq f_{\unif}<\delta\}$, therefore ensures the sign of all instances of $f$ relevant to orbit selection to coincide with the sign of $f_{\unif}$ at the corresponding physical time orbit lengths.
As the sign of $f$ determines the U-turn property, all orbits checked for U-turns during orbit selection have the U-turn property if and only if their physical time orbit length falls within $\{f_{\unif}<0\}$.

The orbit construction doubles the orbits until either the extension satisfies the sub-U-turn property, the extended orbit satisfies the U-turn property, or the maximum number of doublings is reached.
Under the above considerations, the first option does not appear since all sub-orbits checked for U-turns in the sub-U-turn property are of a physical time length already checked for a U-turn in an earlier iteration of the doubling procedure.
Therefore, the recursion terminates when the doubled orbit satisfies the U-turn property, i.e., is of physical time length in $\{f_{\unif}<0\}$, or the maximum number of doublings is reached, in which case the physical time orbit length is $\max\mathfrak T$.
Hence, the orbit selection produces an orbit of physical time length $t_\ast$.
The fact that the doubling procedure extends the orbit to either side with equal probability then yields any orbit in $\mathfrak I_0(t_\ast)$ with uniform probability.
\end{proof}

\begin{figure}[t]
\centering
\begin{tikzpicture}[scale=2]
\clip (-0.2,-1.2) rectangle ({2*pi+0.3},1.2);

\draw[black, line width=1pt, ->] (0,0) -- ({2*pi+0.1},0) node[anchor=north west, pos=1]{$t$};
\draw[black, line width=1pt, ->] (0,-1.1) -- (0,1.1);
\draw[black, line width=1pt] (-0.05,0) -- (0,0) node[anchor=east, pos=0]{${\scriptstyle 0}$};

\draw [gray, line width=1pt] plot[variable=\t,domain=0:2*pi,smooth,thick] ({\t},{sin(\t r)});
\draw [black, line width=1pt, dashed] plot[variable=\t,domain=0:2*pi,smooth,thick] ({\t},{sin(\t r)+0.08});
\draw [black, line width=1pt, dashed] plot[variable=\t,domain=0:2*pi,smooth,thick] ({\t},{sin(\t r)-0.08});

\node (name) at ({3*pi/4+0.1},{sin(3*pi/4 r)+0.2}) [anchor=west]{$f_{\unif}(t)\pm\delta$};

\filldraw[purple] (0.25,0) circle (1.2pt);
\filldraw[purple] (0.5,0) circle (1.2pt);
\filldraw[purple] (1,0) circle (1.2pt);
\filldraw[purple] (2,0) circle (1.2pt);
\filldraw[purple] (4,0) circle (1.2pt) node[anchor=north west]{$t_\ast$};
\end{tikzpicture}
\caption{\textit{Illustration of \Cref{prop:unif_orbit} and its proof.  The red dots depict $\mathfrak T$.  The shown $f_{\unif}=\sin(m^{1/2}t)\,m^{-1/2}d$ arises in isotropic Gaussian distributions, see \eqref{eq:f_conc_iso}.  Assuming $\mathfrak T$ not to intersect the $\delta$-neighborhood around the roots of $f_{\unif}$ and a maximal orbit length $\max\mathfrak T=\Omega(m^{-1/2})$ permitting criticality, NUTS selects orbits of critical physical time length $t_\ast=\Theta(m^{-1/2})$ with high probability.}}
\label{fig:sinm}
\end{figure}
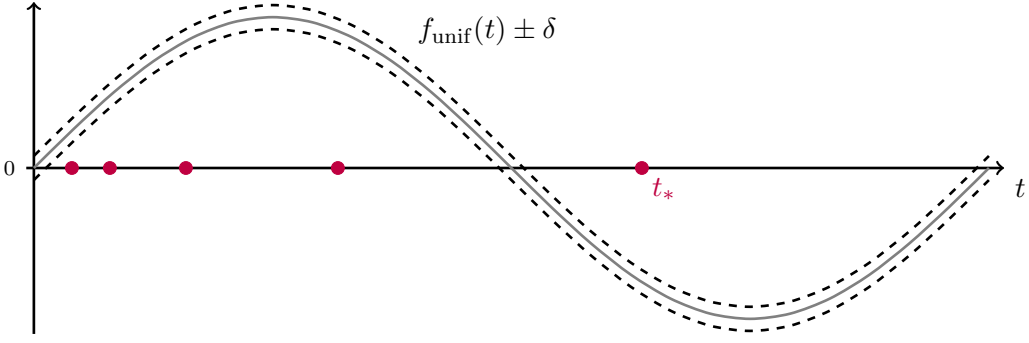

In the remainder of this section, we discuss the concentration of the U-turn property and its implications on orbit selection in specific examples, including isotropic and two-scale Gaussian distributions.

\subsection{Isotropic Gaussian Distributions}\label{sec:iso}

Consider the isotropic Gaussian distribution with covariance matrix $\C=m^{-1}I_d$ for $m>0$.
For notational convenience, we write the distribution as pushforward $m^{-1/2}\#\gamma_d$ of the canonical Gaussian by $x\mapsto m^{-1/2}x$.
Then $m^{-1/2}x\sim m^{-1/2}\#\gamma_d$ for $x\sim\gamma_d$.
Via \eqref{eq:f_deviation}, \Cref{thm:U-turn_concentration} asserts that for $(x,v)\sim (m^{-1/2}\#\gamma_d)\otimes\gamma_d$,
\begin{equation}\label{eq:f_conc_iso}
    f(x,v,t_-,t_+)\ =\ \sin\bigr(m^{1/2}(t_+-t_-)\bigr)\,m^{-1/2}d\ +\ O(m^{-1/2}d^{1/2})
\end{equation}
for all $t_-,t_+\in h\mathbb Z$ with high probability.
In \Cref{lem:deltam} below, we strengthen this result to hold with high probability for all $t_-,t_+$.
For sufficiently large dimension, $f$ tightly concentrates around $f_{\unif}(t)=\sin(m^{1/2}t)\,m^{-1/2}d$, see Figures \ref{fig:sinm} and \ref{fig:Uturn_conc_stdGaussian}.

\begin{figure}
\centering
\includegraphics[scale=.75]{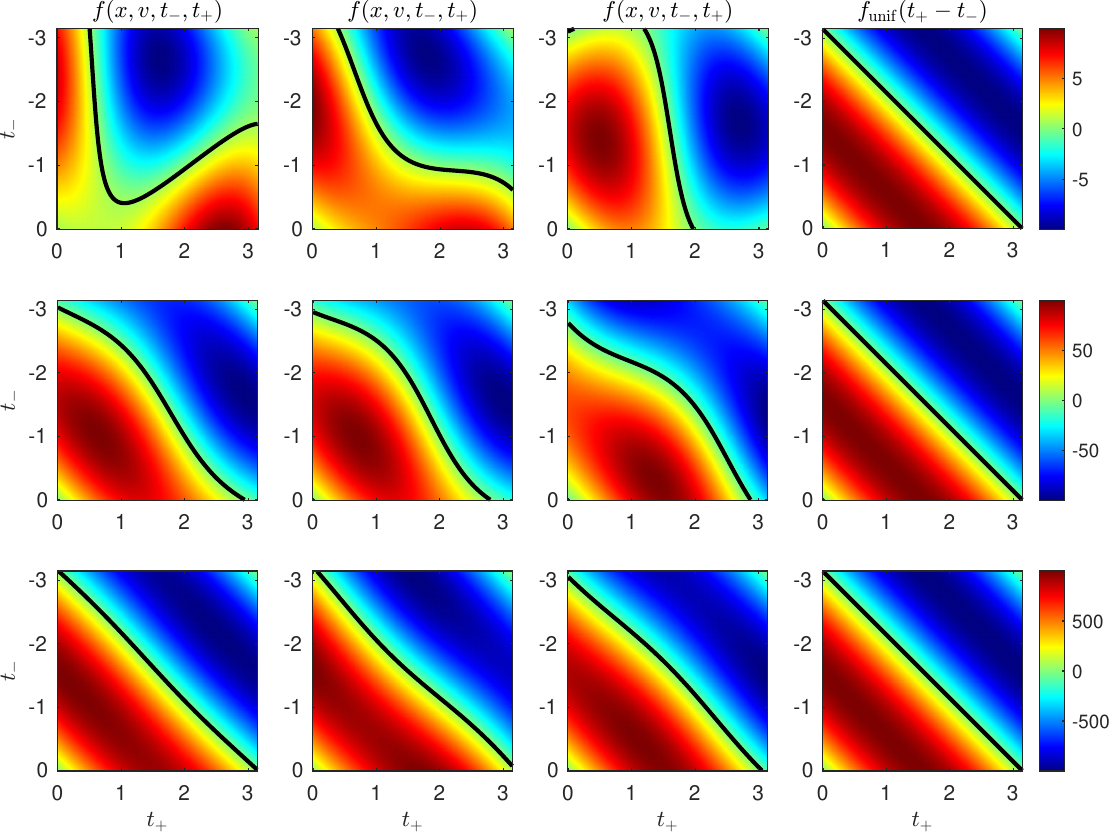}
\caption{\textit{Increasingly tight concentration of the U-turn diagnostic $f$ around $f_{\unif}$ in the canonical Gaussian distribution $\gamma_d$ in dimensions $d=10$ (first row), $d=100$ (second row), and $d=1000$ (third row) according to \eqref{eq:f_conc_iso}.  Respectively, the first three columns depict $f(x,v,t_-,t_+)$ for three independent draws $(x,v)\sim\gamma_{2d}$.  The last column displays the corresponding $f_{\unif}$ for comparison, cf. \Cref{fig:sinm}.  Black lines illustrate zero level sets.  While local effects are pronounced in low dimension, they become neglectable as $d$ increases.}}
\label{fig:Uturn_conc_stdGaussian}
\end{figure}

This concentration results from the geometry of isotropic Gaussian distributions in high dimension.
Specifically, the canonical Gaussian concentrates in spherical shells of the form
\begin{equation}\label{eq:Dalpha}
    D_\alpha^d\ =\ \bigr\{x\in\mathbb R^d\ :\ \bigr||x|^2-d\bigr|\leq\alpha\bigr\}
\end{equation}
in the sense that
\[ \gamma(D_\alpha^d)\ \geq\ 1-2\,\exp\bigr(-\alpha^2/(8d)\bigr)\quad\text{for $\alpha\leq d$.} \]
Accordingly, the isotropic Gaussian $m^{-1/2}\#\gamma_d$ concentrates in the rescaled shell
\[ m^{-1/2}D_\alpha^d\ =\ \bigr\{m^{-1/2}x\ :\ x\in D_\alpha^d\bigr\} \]
since
\begin{equation}\label{eq:Dalphaest}
    (m^{-1/2}\#\gamma_d)(m^{-1/2}D_\alpha^d)\ =\ \gamma(D_\alpha^d)\ \geq\ 1-2\,\exp\bigr(-\alpha^2/(8d)\bigr)\quad\text{for $\alpha\leq d$.}
\end{equation}
This spherical geometry results in \eqref{eq:f_conc_iso} as explained in \Cref{fig:u-turn_iso}.

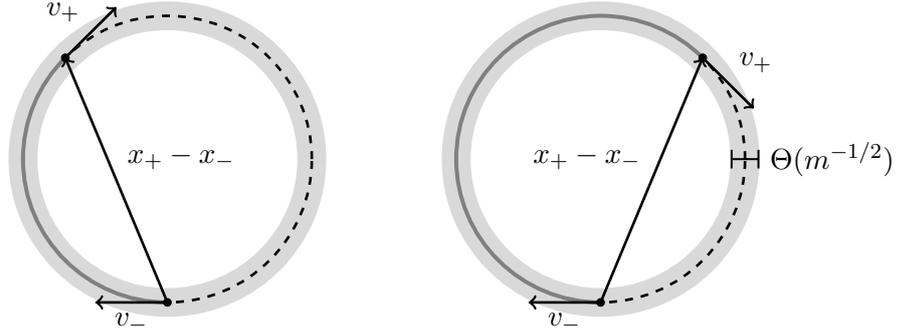
\begin{figure}
\centering
\begin{tikzpicture}[scale=.95]
\clip (-2.5,-0.5) rectangle (12.5,4.5);

\fill [gray!30,even odd rule] (2,2) circle[radius=1.8] circle[radius=2.2];
\fill [gray!30,even odd rule] (8,2) circle[radius=1.8] circle[radius=2.2];

\draw[black, line width=1pt, dashed] (2,2) circle (2);
\draw[black, line width=1pt, dashed] (7+1,2) circle (2);

\draw[gray, line width=1.5pt] (2,0) arc (270:135:2);
\draw[gray, line width=1.5pt] (7+1,0) arc (270:45:2);

\draw[black, line width=1pt, ->] (2,0) -- (1,0) node[anchor=north, pos=0.5]{$v_-$};
\draw[black, line width=1pt, ->] (7+1,0) -- (6+1,0) node[anchor=north, pos=0.5]{$v_-$};

\draw[black, line width=1pt, ->] ({2+2*cos(3*pi/4 r)},{2+2*sin(3*pi/4 r)}) -- ({2+2*cos(3*pi/4 r)+sqrt(2)/2},{2+2*sin(3*pi/4 r)+sqrt(2)/2}) node[anchor=south east, pos=0.5]{$v_+$};
\draw[black, line width=1pt, ->] ({7+2*cos(pi/4 r)+1},{2+2*sin(pi/4 r)}) -- ({7+2*cos(pi/4 r)+sqrt(2)/2+1},{2+2*sin(pi/4 r)-sqrt(2)/2}) node[anchor=south west, pos=0.5]{$v_+$};
    
\draw[black, line width=1pt, ->] (2,0) -- ({2+2*cos(3*pi/4 r)},{2+2*sin(3*pi/4 r)}) node[anchor=south west, pos=0.5]{$x_+-x_-$};
\draw[black, line width=1pt, ->] (7+1,0) -- ({7+2*cos(pi/4 r)+1},{2+2*sin(pi/4 r)}) node[anchor=south east, pos=0.5]{$x_+-x_-$};

\filldraw (2,0) circle (1.5pt);
\filldraw ({2+2*cos(3*pi/4 r)},{2+2*sin(3*pi/4 r)}) circle (1.5pt);
\filldraw (7+1,0) circle (1.5pt);
\filldraw ({7+2*cos(pi/4 r)+1},{2+2*sin(pi/4 r)}) circle (1.5pt);

\draw [|-|, line width=.75pt] (10-0.2,2) -- (10+0.2,2) node[right] {$\Theta(m^{-1/2})$};
\end{tikzpicture}
\caption{\textit{Virtually all mass of $m^{-1/2}\#\gamma_d$ concentrates in a spherical shell $m^{-1/2}D_\alpha^d$ (shaded gray) with $\alpha=\Theta(d^{1/2})$ around a centered sphere (dashed), see \eqref{eq:Dalphaest}.  With increasing dimension, the deviations of order $m^{-1/2}$ from the sphere become neglectable relative to its radius $(d/m)^{1/2}$.  Hamiltonian orbits obey this spherical geometry and hence increasingly resemble the $(2\pi m^{-1/2})$-periodic Hamiltonian orbits restricted to the sphere, whose U-turn diagnostic precisely coincides with $f_{\unif}$ resulting in \eqref{eq:f_conc_iso}.  In particular, orbits restricted to the sphere of physical time length in $m^{-1/2}[0,\pi]$ do not have the U-turn property (left) while orbits of length in $m^{-1/2}(\pi,2\pi)$ do (right).}}
\label{fig:u-turn_iso}
\end{figure}

We now strengthen \eqref{eq:f_conc_iso}.
Therefore, define
\begin{equation}\label{eq:E}
    E_{\alpha,r}^d\ =\ \Bigr\{v\in\mathbb R^d\ :\ \max\bigr(||v|^2-d|,\,\sup\nolimits_{x\in D_\alpha^d}|x\cdot v|\bigr)\ \leq\ r\Bigr\}\;.
\end{equation}
By \cite[Lemma~1]{BouRabeeOberdoersterNUTS1}, for $0\leq\alpha,r\leq d$, it holds that
\begin{equation}\label{eq:Eprob}
    \gamma_d(E_{\alpha,r}^d)\ \geq\ 1-4\,e^{-r^2/(8d)}\;.
\end{equation}
In particular, the product measure $(m^{-1/2}\#\gamma_d)\otimes\gamma_d$ on phase space concentrates within $(m^{-1/2}D_\alpha^d)\times E_{\alpha,r}^d$ with $\alpha,r=\Theta(d^{1/2})$.

\begin{lemma}\label{lem:deltam}
Consider $\mu=m^{-1/2}\#\gamma_d$ for $m>0$ and $d\in\mathbb N$.
Let $\alpha,r\leq d$ and for $\hbar\geq0$ set
\begin{equation}\label{eq:delta_general}
    \delta_{m,d,\hbar}(\alpha,r)\ =\ \bigr(5\max(\alpha,r)d^{-1/2}+\hbar^2md^{1/2}\bigr)m^{-1/2}d^{1/2}
\end{equation}
as well as
\begin{equation}\label{eq:funif_general}
    f_{\unif,\hbar}(t)\ =\ \sin\bigr(\beta_{\hbar^2m}m^{1/2}t\bigr)\,m^{-1/2}d
\end{equation}
with $\beta_{\mathsf{x}}=\mathsf{x}^{-1/2}\arccos(1-\mathsf{x}/2)$.
Further let $(x,v)\in (m^{-1/2}D_\alpha^d)\times E_{\alpha,r}^d$ and
\[ \Xi\ \in\ \Bigr\{v_+\cdot(x_+-x_-),\,v_-\cdot(x_+-x_-),\,f(x,v,t_-,t_+)\Bigr\} \]
with $(x_-,v_-)$, $(x_+,v_+)$ and $f$ as in \Cref{def:f} using
\begin{itemize}
\item[(i)] Hamiltonian flow, or
\item[(ii)] leapfrog flow with step-size $h>0$ such that $h^2m\leq1$.
\end{itemize}
Then, it holds
\begin{equation}\label{eq:min_m}
    \bigr|\Xi\ -\ f_{\unif,\hbar}(t_+-t_-)\bigr|\ \leq\ \delta_{m,d,\hbar}(\alpha,r)
\end{equation}
\begin{itemize}
\item[(i)] with $\hbar=0$ for all $t_-,t_+\in \mathbb R$ in case of Hamiltonian flow, and 
\item[(ii)] with $\hbar=h$ for all $t_-,t_+\in h\mathbb Z$ in case of leapfrog flow.
\end{itemize}
\end{lemma}

Before proving the lemma, we discuss its implications on orbit selection according to \Cref{prop:unif_orbit}.
Since $(x,v)\sim(m^{-1/2}\#\gamma_d)\otimes\gamma_d$ is contained in $(m^{-1/2}D_\alpha^d)\times E_{\alpha,r}^d$ with $\alpha,r=\Theta(d^{1/2})$ with high probability, the lemma ensures \eqref{eq:prop_deviation} holds with $f_{\unif,\hbar}$ and $\delta_{m,d,\hbar}(\alpha,r)$ for all $t_-,t_+$.
Assume \eqref{eq:prop_T} holds.
The proposition then asserts the orbit selection in NUTS to be uniform from the orbits $\mathfrak I_0(t_\ast)$ of physical time length
\begin{equation}\label{eq:tastiso}
    t_\ast\ =\ \inf\bigr\{t\in\mathfrak T\ :\ f_{\unif,\hbar}(t)<0\bigr\}\ \wedge\ \max\mathfrak T\;.
\end{equation}
Further assuming the maximal orbit length $\max\mathfrak T=\Omega(m^{-1/2})$ to permit criticality, $t_\ast$ is guaranteed to be of order $m^{-1/2}$.
This can be seen in two steps:
\begin{enumerate}
\item On one hand, $t_\ast=\Omega(m^{-1/2})$ as both terms in \eqref{eq:tastiso} are $\Omega(m^{-1/2})$ since
\begin{align*}
    \inf\bigr\{t\in\mathfrak T\ :\ f_{\unif,\hbar}(t)<0\bigr\}\ &\geq\ \inf\bigr\{t\geq0\ :\ f_{\unif,\hbar}(t)<0\bigr\}\ =\ \pi\beta_{\hbar^2m}^{-1}m^{-1/2}\\
    &=\ \Omega(m^{-1/2})\;.
\end{align*}
\item On the other, $t_\ast=O(m^{-1/2})$ as by \eqref{eq:prop_T},
\[ t_\ast\ \leq\ \inf\bigr\{t\in\mathfrak T\ :\ f_{\unif,\hbar}(t)<0\bigr\}\ <\ 2\pi\beta_{\hbar^2m}^{-1}m^{-1/2}\ =\ O(m^{-1/2})\;. \]
\end{enumerate}
NUTS thus selects orbits of physical time length $\Theta(m^{-1/2})$, cf. \Cref{fig:sinm}.
As the resulting integration times precisely correspond to criticality in Randomized HMC, we expect NUTS to achieve the diffusive-to-ballistic speed-up in isotropic Gaussian distributions of sufficiently high dimension for the U-turn concentration to be pronounced, cf. \Cref{fig:Uturn_conc_stdGaussian}.
Rigorously, this can be shown via the coupling approach to mixing of NUTS presented in \textsection\ref{sec:mixing}.
Our observations are summarized in the following proposition.

\begin{proposition}
Consider $\mu=m^{-1/2}\#\gamma_d$ for $m>0$ and $d\in\mathbb N$.
Let $\alpha,r\leq d$, $\hbar\geq0$, and $f_{\unif,\hbar}$ and $\delta_{m,d,\hbar}$ as defined in \eqref{eq:delta_general} and \eqref{eq:funif_general}.
Assume \eqref{eq:prop_T} holds with $f_{\unif,\hbar}$ and $\delta_{m,d,\hbar}$
\begin{itemize}
\item[(i)] with $\hbar=0$ in case of NUTS using Hamiltonian flow, and
\item[(ii)] with $\hbar=h$ in case of NUTS using leapfrog flow with step-size $h>0$ such that $h^2m\leq1$.
\end{itemize}
If $\max\mathfrak T=\Omega(m^{-1/2})$, then for all $(x,v)\in (m^{-1/2}D_\alpha^d)\times E_{\alpha,r}^d$, the orbit selection of NUTS satisfies
\[ \mathcal{O}_{x,v}\ =\ \Unif(\mathfrak I_0(t_\ast))\quad\text{with $t_\ast=\Theta(m^{-1/2})$.} \]
\end{proposition}

It remains to prove the lemma.
\begin{proof}[Proof of \Cref{lem:deltam}]
We prove \eqref{eq:min_m} for leapfrog flow.
By taking the limit $h\to0$, the assertion for Hamiltonian flow follows analogously.
Let $m>0$.
Leapfrog flow corresponding to the Hamiltonian
\[ H_m(x,v)\ =\ \frac m2|x|^2+\frac12|v|^2 \]
with step size $h>0$, initial condition $(x,v)\in\mathbb R^{2d}$ and integration time $t\in h\mathbb Z$ takes the form
\begin{equation}\label{eq:mLF}
    \begin{aligned}[t]
    \Phi_{h,m}^{t/h}(x,v)\ &=\ \Bigr(\begin{aligned}[t]
        &\cos_{h,m}(m^{1/2}t)\,x\ +\ \sin_{h,m}(m^{1/2}t)(1-h^2m/4)^{-1/2} m^{-1/2}\,v,\\
        &-\sin_{h,m}(m^{1/2}t)(1-h^2m/4)^{1/2}m^{1/2}\,x\ +\ \cos_{h,m}(m^{1/2}t)\,v\Bigr)
        \end{aligned}
    \end{aligned}
\end{equation}
where
\[ \cos_{h,m}(m^{1/2}t)\ =\ \cos(\beta_{h^2m}m^{1/2}t)\quad\text{and}\quad\sin_{h,m}(m^{1/2}t)\ =\ \sin(\beta_{h^2m}m^{1/2}t) \]
with $\beta_{\mathsf{x}}=\arccos(1-\mathsf{x}/2)/\sqrt{\mathsf{x}}=1+O(\mathsf{x})$ as $\mathsf{x}\searrow 0$.
With the projections defined in \eqref{eq:proj}, we write
\[ \Phi_{h,m}^{t/h}(x,v)\ =\ \bigr(\proj_1^d\Phi_{h,m}^{t/h}(x,v),\,\proj_2^d\Phi_{h,m}^{t/h}(x,v)\bigr)\;. \]

Let $\alpha,r\leq d$ and $(x,v)\in (m^{-1/2}D_\alpha^d)\times E_{\alpha,r}^d$.
By \eqref{eq:Dalpha} and \eqref{eq:E},
\begin{equation}\label{eq:Econtrolm}
    ||m^{1/2}x|-d|\ \leq\ \alpha\quad\text{and}\quad\max\bigr(||v|^2-d|,\,\sup\nolimits_{x\in m^{-1/2}D_\alpha^d}|m^{1/2}x\cdot v|\bigr)\ \leq\ r\;.
\end{equation}
Let $t_+,t_-\in h\mathbb Z$.
Inserting \eqref{eq:mLF} into
\[ v_+\cdot(x_+-x_-)\ =\ \proj_2^{d}\Phi_{h,m}^{t_+/h}(x,v)\cdot\bigr(\proj_1^{d}\Phi_{h,m}^{t_+/h}(x,v)-\proj_1^{d}\Phi_{h,m}^{t_-/h}(x,v)\bigr) \]
and applying the triangle inequality yields
\begin{align*}
\Bigr|\frac{v_+\cdot(x_+-x_-)}{m^{-1/2}d}-\sin_{h,m}\bigr(m^{1/2}(t_+-t_-)\bigr)\Bigr| \ \leq \ \rn{1} + \rn{2} + \rn{3} + \rn{4}
\end{align*}
where we introduced

\scalebox{.87}{\parbox{1.1\linewidth}{%
\begin{align*}
\rn{1} \ &=\ \left|\left( \frac12\sin_{h,m}(2m^{1/2}t_+)-\sin_{h,m}(m^{1/2}t_+)\cos_{h,m}(m^{1/2}t_-)\right)(1-h^2m/4)^{1/2}\frac{|m^{1/2}x|^2-d}{d} \right| \;, \\
\rn{2} \ &=\ \left| \left(\frac12\sin_{h,m}(2m^{1/2}t_+)-\cos_{h,m}(m^{1/2}t_+)\sin_{h,m}(m^{1/2}t_-) \right)(1-h^2m/4)^{-1/2}\frac{|v|^2-d}{d}\right| \;, \\
\rn{3} \ &= \  \left| \left(\frac12\sin_{h,m}(2m^{1/2}t_+)-\sin_{h,m}(m^{1/2}t_+)\cos_{h,m}(m^{1/2}t_-)\right)\bigr(1-(1-h^2m/4)^{1/2}\bigr)\right.\\
& \qquad \left. +\left(\frac12\sin_{h,m}(2m^{1/2}t_+)-\cos_{h,m}(m^{1/2}t_+)\sin_{h,m}(m^{1/2}t_-)\right)\bigr((1-h^2m/4)^{-1/2}-1\bigr)\right|\;, \\
\rn{4} \ &= \ \left| \cos(2\beta_{h^2m}m^{1/2}t_+)-\cos\bigr(\beta_{h^2m}m^{1/2}(t_++t_-)\bigr) \right| \left|\frac{m^{1/2}x\cdot v}{d} \right|\;.
\end{align*}
}}

Using the elementary bounds $\left|\sin(\mathsf{a})/2-\sin(\mathsf{b}) \cos(\mathsf{c})\right| \le 3/2$ valid for all $\mathsf{a},\mathsf{b},\mathsf{c} \in \mathbb{R}$ and $\left|\cos(\mathsf{a})-\cos(\mathsf{b})\right| \le 2$ valid for all $\mathsf{a},\mathsf{b} \in \mathbb{R}$, and subsequently inserting \eqref{eq:Econtrolm}, $r\leq d$ as well as $h^2m\leq1$, we obtain
\begin{align}\label{eq:uturn_rn2to5}
\rn{1}+\rn{2} + \rn{3} + \rn{4} \ & \le \ \frac32\left|\frac{|m^{1/2}x|^2-d}{d}\right|+\frac32(1-h^2m/4)^{-1/2}\left|\frac{|v|^2-d}{d}\right| \nonumber\\
& \quad +\frac32\bigr((1-h^2m/4)^{-1/2}-(1-h^2m/4)^{1/2}\bigr) + 2\left|\frac{m^{1/2}x\cdot v}{d}\right|\nonumber\\
&\leq\ \frac32\frac{\alpha+r}{d}+\frac32\bigr(2(1-h^2m/4)^{-1/2}-(1-h^2m/4)^{1/2}-1\bigr)+2\frac{r}{d}\nonumber\\
&\leq\ 5\max(\alpha,r)d^{-1}+\frac34h^2m\;.
\end{align}
An analogous bound holds for the second dot product and hence also for the minimum of the two.
\end{proof}

\subsection{Two-scale Gaussian Distributions}\label{sec:2S}

To go beyond isotropy, we now consider two-scale Gaussian distributions
\begin{equation}\label{eq:2S}
    \gamma^{2S}\ =\ (m_1^{-1/2}\#\gamma_{d_1})\ \otimes\ (m_2^{-1/2}\#\gamma_{d_2})\quad\text{on $\mathbb R^{d_1+d_2}$}
\end{equation}
for $0<m_1\leq m_2$ and $d_1,d_2\in\mathbb N$.
Let $d=d_1+d_2$ denote the total dimension and $\kappa=m_2/m_1$ the condition number.

Via \eqref{eq:f_deviation}, \Cref{thm:U-turn_concentration} asserts that for $(x,v)\sim \gamma^{2S}\otimes\gamma_d$,
\begin{equation}\label{eq:f_conc_2S}
    f(x,v,t_-,t_+)\ =\ f^{2S}_{\unif}(t_+-t_-)\ +\ O\bigr((m_1^{-1}d_1+m_2^{-1}d_2)^{1/2}\bigr)
\end{equation}
with
\begin{equation}\label{eq:f2Sunif}
    f^{2S}_{\unif}(t)\ =\ \sin\bigr(m_1^{1/2}(t_+-t_-)\bigr)\,m_1^{-1/2}d_1\ +\ \sin\bigr(m_2^{1/2}(t_+-t_-)\bigr)\,m_2^{-1/2}d_2
\end{equation}
for all $t_-,t_+\in h\mathbb Z$ with high probability, see \Cref{fig:sin2S}.
Again, we strengthen this result later on to hold with high probability for all $t_-,t_+$.

Note that $f^{2S}_{\unif}$ is the sum of the uniform terms arising in the two isotropic factors of the two-scale distribution.
This results from the geometric structure of the two-scale distribution.
Specifically, $\gamma^{2S}$ concentrates in
\begin{equation}\label{eq:Dalpha_2S}
    D^{2S}_{\alpha}\ =\ (m_1^{-1/2}D_{\alpha_1}^{d_1})\ \times\ (m_2^{-1/2}D_{\alpha_2}^{d_2})\quad\text{for $\alpha=(\alpha_1,\alpha_2)$}
\end{equation}
according to
\begin{equation}\label{eq:D2Sbound}
    \gamma^{2S}(D^{2S}_{\alpha})\ \geq\ 1-4\,\exp\Bigr(-\frac18\min(\alpha_1^2/d_1,\,\alpha_2^2/d_2)\Bigr)\;,
\end{equation}
which follows from \eqref{eq:Dalphaest}.
$D^{2S}_{\alpha}$ consists of spherical shells in the individual factors of the two-scale distribution, in which the motion reduces to the isotropic case discussed above.
The dot products in $f(x,v,t_-,t_+)$ then reduce to sums over the two factors.

\begin{figure}[t]
\centering
\begin{tikzpicture}[scale=2]
\clip (-0.2,-1.5) rectangle ({2*pi+0.8},1.5);

\draw[black, line width=1pt, ->] (0,0) -- ({2*pi+0.2},0) node[anchor=north west, pos=1]{$t$};
\draw[black, line width=1pt, ->] (0,-1.4) -- (0,1.4);
\draw[black, line width=1pt] (-0.05,0) -- (0,0) node[anchor=east, pos=0]{${\scriptstyle 0}$};

\draw [gray, line width=1pt] plot[variable=\t,domain=0:(2*pi+0.2),samples=100,smooth,thick] ({\t},{sin(\t r)+sqrt(100)^(-1/2)*sin(sqrt(100)*\t r)});
\draw [black, line width=1pt, dashed] plot[variable=\t,domain=0:(2*pi+0.2),samples=100,smooth,thick] ({\t},{sin(\t r)+sqrt(100)^(-1/2)*sin(sqrt(100)*\t r)+0.08});
\draw [black, line width=1pt, dashed] plot[variable=\t,domain=0:(2*pi+0.2),samples=100,smooth,thick] ({\t},{sin(\t r)+sqrt(100)^(-1/2)*sin(sqrt(100)*\t r)-0.08});
\node (name) at ({3*pi/4},1.25) [anchor=west]{$f_{\unif}^{2S}(t)\pm\delta$};

\filldraw[purple] (0.15,0) circle (1.2pt);
\filldraw[purple] (0.3,0) circle (1.2pt);
\filldraw[purple] (0.6,0) circle (1.2pt);
\filldraw[purple] (1.2,0) circle (1.2pt);
\filldraw[purple] (2.4,0) circle (1.2pt);
\filldraw[purple] (4.8,0) circle (1.2pt) node[anchor=north west]{$t_\ast$};

\end{tikzpicture}
\caption{\textit{
Uniform term in two-scale Gaussian distributions, cf. \Cref{fig:sinm}.
As $d_1,d_2\to\infty$, the deviations $\delta$ vanish relative to the scale of $f^{2S}_{\unif}$.
The distribution falls within the accelerated phase $\mathfrak A$ if and only if $f^{2S}_{\unif}$ remains non-negative during the first period of the faster oscillating sine.
In the vicinity of the roots of the slower oscillating term, the sign of the uniform term is predominately determined by the faster oscillating term.
}}
\label{fig:sin2S}
\end{figure}

To strengthen \eqref{eq:f_conc_2S}, define
\begin{equation}\label{eq:E2S}
    E^{2S}_{\alpha,r}\ =\ E_{\alpha_1,r_1}^{d_1}\ \times\ E_{\alpha_2,r_2}^{d_2}\quad\text{for $\alpha=(\alpha_1,\alpha_2)$ and $r=(r_1,r_2)$}
\end{equation}
which by \eqref{eq:Eprob}, for $\alpha_1,r_1\leq d_1$ and $\alpha_2,r_2\leq d_2$, satisfies
\begin{equation}\label{eq:Eprob_2S}
    \gamma_d(E^{2S}_{\alpha,r})\ \geq\ 1-8\exp\Bigr(-\frac18\min(r_1^2/d_1,\,r_2^2/d_2)\Bigr)\;.
\end{equation}

\begin{lemma}\label{lem:delta2S}
Consider $\mu=\gamma^{2S}$ as defined in \eqref{eq:2S} with $0<m_1\leq m_2$ and $d_1,d_2\in\mathbb N$.
Let $\alpha=(\alpha_1,\alpha_2)$ and $r=(r_1,r_2)$ such that $\alpha_1,r_1\leq d_1$ and $\alpha_2,r_2\leq d_2$.
Further let $\hbar\geq0$ and set
\begin{equation}\label{eq:delta}
    \delta^{2S}_\hbar(\alpha,r)\ =\ \delta_{m_1,d_1,\hbar}(\alpha_1,r_1)\ +\ \delta_{m_2,d_2,\hbar}(\alpha_2,r_2)
\end{equation}
with $\delta_{m,d,\hbar}$ as defined in \eqref{eq:delta_general}, as well as
\begin{equation}\label{eq:f}
    f^{2S}_{\unif,\hbar}(t)\ =\ \sin\bigr(\beta_{\hbar^2m_1}m_1^{1/2}t\bigr)\,m_1^{-1/2}d_1\ +\ \sin\bigr(\beta_{\hbar^2m_2}m_2^{1/2}t\bigr)\,m_2^{-1/2}d_2\;.
\end{equation}
Suppose $(x,v)\in D_\alpha^{2S}\times E_{\alpha,r}^{2S}$ and $f$ as in \Cref{def:f} using
\begin{itemize}
\item[(i)] Hamiltonian flow, or
\item[(ii)] leapfrog flow with step-size $h>0$ such that $h^2m_2\leq1$.
\end{itemize}
Then, it holds
\begin{equation}\label{eq:min_2S}
    \bigr|f(x,v,t_-,t_+)\ -\ f^{2S}_{\unif,\hbar}(t_+-t_-)\bigr|\ \leq\ \delta^{2S}_\hbar(\alpha,r)
\end{equation}
\begin{itemize}
\item[(i)] with $\hbar=0$ for all $t_-,t_+\in\mathbb R$ in case of Hamiltonian flow, and
\item[(ii)] with $\hbar=h$ for all $t_-,t_+\in h\mathbb Z$ in case of leapfrog flow.
\end{itemize}
\end{lemma}

Before we discuss its implications on orbit selection, we show how the lemma results from applying \Cref{lem:deltam} to the individual isotropic factors of the two-scale distribution.

\begin{proof}[Proof of \Cref{lem:delta2S}]
We again show \eqref{eq:min_2S} for leapfrog flow while the result for Hamiltonian flow follows in the limit $h\to0$.
Define the norm adapted to the two-scale geometry by
\begin{equation}\label{eq:2Snorm}
    |x|_{2S}^2\ =\ m_1|x^1|^2+m_2|x^2|^2\quad\text{for $x=(x^1,x^2)\in\mathbb R^{d_1+d_2}$.}
\end{equation}
Leapfrog flow with respect to the two-scale Hamiltonian
\begin{equation}\label{eq:H2S}
    H_{2S}(x,v)\ =\ \frac12|x|_{2S}^2 + \frac12|v|^2
\end{equation}
with step-size $h>0$, $t\in h\mathbb Z$, and $(x,v)=(x^1,x^2,v^1,v^2)\in\mathbb R^{2(d_1+d_2)}$ is a combination of leapfrog flow \eqref{eq:mLF} in the two isotropic factors of the two-scale distribution
\begin{align}\label{eq:twoscaleLF}
    \Phi_{h,2S}^{t/h}(x,v)\ =\ \Bigr(
    &\proj_1^{d_1}\Phi_{h,m_1}^{t/h}(x^1,v^1),\,\proj_1^{d_2}\Phi_{h,m_2}^{t/h}(x^2,v^2),\nonumber\\
    &\proj_2^{d_1}\Phi_{h,m_1}^{t/h}(x^1,v^1),\,\proj_2^{d_2}\Phi_{h,m_2}^{t/h}(x^2,v^2)\Bigr)\;.
\end{align}
Let $(x_+,v_+)=(x^1_+,x^2_+,v^1_+,v^2_+)$ and $(x_-,v_-)=(x^1_-,x^2_-,v^1_-,v^2_-)$.
The first dot product in $f$ decomposes into the two scales as
\begin{equation}\label{eq:fac}
    v_+\cdot(x_+-x_-)\ =\ v^1_+\cdot(x^1_+-x^1_-)\ +\ v^2_+\cdot(x^2_+-x^2_-)\;.
\end{equation}
Applying \Cref{lem:deltam} to the individual terms yields
\begin{equation*}
\begin{aligned}[t]
&\Bigr|v^1_+\cdot(x^1_+-x^1_-)m_1^{1/2}d_1^{-1}\ -\ \sin\bigr(\beta_{h^2m_1}m_1^{1/2}(t_+-t_-)\bigr)\Bigr|\ \leq\ \delta_{m_1,d_1,h}(\alpha_1,r_1)\;, \\
&\Bigr|v^2_+\cdot(x^2_+-x^2_-)m_2^{1/2}d_2^{-1}\ -\ \sin\bigr(\beta_{h^2m_2}m_2^{1/2}(t_+-t_-)\bigr)\Bigr|\ \leq\ \delta_{m_2,d_2,h}(\alpha_2,r_2)\;.
\end{aligned}
\end{equation*}
Inserting into \eqref{eq:fac} shows
\[ \bigr|v_+\cdot(x_+-x_-)\ -\ f^{2S}_{\unif,h}(t_+-t_-)\bigr|\ \leq\ \delta_{m_1,d_1,h}(\alpha_1,r_1)\ +\ \delta_{m_2,d_2,h}(\alpha_2,r_2)\ =\ \delta^{2S}_h(\alpha,r)\;.
 \]
An analogous bound holds for the second dot product in $f$ and hence for $f$ itself.
\end{proof}

Let us discuss the lemma's implications on orbit selection according to \Cref{prop:unif_orbit}.
Since $(x,v)\sim\gamma^{2S}\otimes\gamma_d$ is contained in $D_\alpha^{2S}\times E_{\alpha,r}^{2S}$ with $\alpha_1,r_1=\Theta(d_1^{1/2})$ and $\alpha_2,r_2=\Theta(d_2^{1/2})$ with high probability, the lemma ensures \eqref{eq:prop_deviation} holds with $f^{2S}_{\unif,\hbar}$ and $\delta^{2S}_\hbar(\alpha,r)$ for all $t_-,t_+$.
Assume \eqref{eq:prop_T} holds.
The proposition then asserts the orbit selection in NUTS to be uniform from the orbits $\mathfrak I_0(t_\ast)$ of physical time length
\begin{equation}\label{eq:tast2S}
    t_\ast\ =\ \inf\bigr\{t\in\mathfrak T\ :\ f^{2S}_{\unif,\hbar}(t)<0\bigr\}\ \wedge\ \max\mathfrak T\;.
\end{equation}
Since the two-scale distribution satisfies the Poincar\'e inequality \eqref{eq:poincare} with $m=\Theta(m_1)$, the integration times of critical Randomized HMC are of order $m_1^{-1/2}$.
This enables global moves in both scales, in contrast to qualitatively shorter integration times which restrict to local moves in the spread-out $m_1$-scale, cf. \Cref{fig:Gaussian_acceleration}.
In order to permit critical orbit lengths in NUTS, we therefore assume a maximal orbit length $\max\mathfrak T=\Omega(m_1^{-1/2})$.
For simplicity, we restrict our attention to Hamiltonian flow, in which case $\hbar=0$ and $f^{2S}_{\unif,0}=f^{2S}_{\unif}$ as in \eqref{eq:f2Sunif}.
Note however that the step size $h$ remains a parameter as it determines the physical time spacing between consecutive orbit elements.

For all $h$, which determine $\mathfrak T$ via \eqref{eq:frakT}, it holds
\begin{equation}\label{eq:ineqeq}
    \inf\bigr\{t\in\mathfrak T\ :\ f^{2S}_{\unif}(t)<0\bigr\}\ \geq\ \inf\bigr\{t\geq0\ :\ f^{2S}_{\unif}(t)<0\bigr\}\;.
\end{equation}
The right hand side is either contained in $m_2^{-1/2}(0,2\pi)$ if the faster oscillating sine term in $f^{2S}_{\unif}$ produces negative values within that interval, or $\Omega(m_1^{-1/2})$ since otherwise the slower oscillating sine term dominates forcing positive values up to the vicinity of its root in $\pi m_1^{-1/2}$, see \Cref{fig:sin2S}.
Since $f_{\unif}^{2S}(t)<0$ for some $t\in m_2^{-1/2}(0,2\pi)$ if and only if
\[ g_{\kappa,\,d_2/d_1}(t)\ =\ \sin(\kappa^{-1/2}t)+\sin(t)\kappa^{-1/2}d_2/d_1\ <\ 0\quad\text{for some $t\in(0,2\pi)$,} \]
this distinction solely depends on the relation between the condition number $\kappa=m_2/m_1$ and $d_2/d_1$.

Define the accelerated phase
\begin{equation}\label{eq:BP}
    \mathfrak A\ =\ \bigr\{(\kappa,\,d_2/d_1)\in[1,\infty)\times(0,\infty)\ :\ g_{\kappa,\,d_2/d_1}(t)\geq0\text{ for all $t\in(0,2\pi)$}\bigr\}
\end{equation}
depicted in \Cref{fig:PT} and, abusing notation, write $\gamma^{2S}\in\mathfrak A$ iff the distributions parameters $(\kappa,\,d_2/d_1)\in\mathfrak A$.
In particular, we observe the following dichotomy:
If $\gamma^{2S}\in\mathfrak A$, it holds $t_\ast=\Omega(m_1^{-1/2})$ for all $h$ permitted under \eqref{eq:prop_T}.
If on the other hand $\gamma^{2S}\notin\mathfrak A$, the right hand side of \eqref{eq:ineqeq} is contained in $m_2^{-1/2}(0,2\pi)$.
In case there exists a permitted step size $h$ such that \eqref{eq:ineqeq} becomes an equality, we then encounter qualitatively shorter orbits of length $t_\ast=\Theta(m_2^{-1/2})$ for such $h$.
The existence of such a step size is however not guaranteed for all $\gamma^{2S}\notin\mathfrak A$ since, especially in low-dimensional settings where concentration of the U-turn property is not pronounced, \eqref{eq:prop_T} might be severe, excluding a wide range of step sizes.
Nevertheless, as the asymptotic study following the next proposition shows, increasing $d_1$ and $d_2$ weakens \eqref{eq:prop_T} so that we can assert the existence of $\gamma^{2S}\notin\mathfrak A$ and a step size such that $t_\ast=\Theta(m_2^{-1/2})$.

\begin{figure}
\centering
\begin{tikzpicture}[scale=0.1]
\draw[line width=1pt] plot[variable=\t,domain=2.14:100,samples=100,smooth,thick] ({\t},{10*sqrt(\t)*sin(4.6/sqrt(\t) r)});
\fill [gray!20, domain=2.14:100,samples=100,smooth, variable=\t]
  (2.14, 0)
  -- plot ({\t},{10*sqrt(\t)*sin(4.6/sqrt(\t) r)})
  -- ({100},{10*sqrt(100)*sin(4.6/sqrt(100) r)})
  -- (100,0)
  -- cycle;
\draw[thick] (2.14,0) -- (2.14,60);
\fill [gray!20]
  (2.14, 0) -- (2.14,60) -- (-1,60) -- (-1,0) -- cycle;
\draw[dashed] (-1,46) -- (100,46);
\draw[->,line width=1pt] (-1,0) -- (100,0) node[pos=1,below] {$\kappa$};
\draw[->,line width=1pt] (-1,0) -- (-1,60) node[pos=1,left] {$d_2/d_1$};
\node at (55,20) {Accelerated Phase $\mathfrak A$};
\draw[line width=1pt] (90,0) -- (90,-1) node[below] {$100$};
\draw[line width=1pt] (-1,46) -- (-2,46) node[left] {$4.6\approx a$};
\draw[line width=1pt] (-1,0) -- (-1,-1) node[below] {$1$};
\draw[line width=1pt] (2.14,0) -- (2.14,-1) node[below] {$4$};
\draw[line width=1pt] (-1,0) -- (-2,0) node[left] {$0$};
\end{tikzpicture}
\caption{\textit{
The accelerated phase.
For $1\leq\kappa<4$, $(\kappa,\,d_2/d_1)\in\mathfrak A$ for all $d_2/d_1>0$.
As $\kappa\to\infty$, $\sin(\kappa^{-1/2}t)\approx\kappa^{-1/2}t$ becomes an increasingly accurate approximation for $t\in(0,2\pi)$, so that the phase transition asymptotically occurs at $d_2/d_1=a$ with $a=\inf\{q>0:q\sin t+t=0\text{ for some $t\in(0,2\pi)$}\}\approx4.6$.
In between, the transition approximately occurs at $\kappa^{1/2}\sin(\kappa^{-1/2}a)$.
}}
\label{fig:PT}
\end{figure}
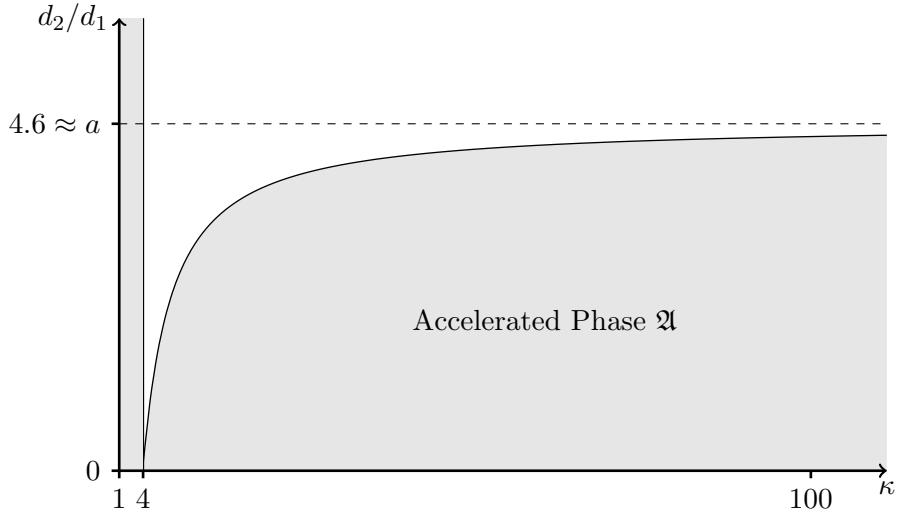

Our observations are summarized in the following proposition.

\begin{proposition}\label{prop:2S}
Consider $\mu=\gamma^{2S}$ as defined in \eqref{eq:2S} with $0<m_1\leq m_2$ and $d_1,d_2\in\mathbb N$.
Let $\alpha=(\alpha_1,\alpha_2)$, $r=(r_1,r_2)$ such that $\alpha_1,r_1\leq d_1$, $\alpha_2,r_2\leq d_2$.
Further let $\hbar\geq0$, and $f^{2S}_{\unif,\hbar}$ and $\delta^{2S}_\hbar$ be as defined in \eqref{eq:f} and \eqref{eq:delta}.
Assume \eqref{eq:prop_T} holds with $f^{2S}_{\unif,\hbar}$ and $\delta^{2S}_\hbar$
\begin{itemize}
\item[(i)] with $\hbar=0$ in case of NUTS using Hamiltonian flow, and 
\item[(ii)] with $\hbar=h$ in case of NUTS using leapfrog flow with step-size $h>0$ such that $h^2m_2\leq1$.
\end{itemize}
Then, for all $(x,v)\in D_\alpha^{2S}\times E_{\alpha,r}^{2S}$, the orbit selection of NUTS satisfies
\[ \mathcal{O}_{x,v}\ =\ \Unif(\mathfrak I_0(t_\ast))\quad\text{with $t_\ast$ as in \eqref{eq:tast2S}.} \]
For NUTS using Hamiltonian flow, assuming $\max\mathfrak T=\Omega(m_1^{-1/2})$, the following dichotomy holds:
\begin{itemize}
\item For all $\gamma^{2S}\in\mathfrak A$ and all permitted step sizes $h$, it holds $t_\ast=\Omega(m_1^{-1/2})$.
\item There exist $\gamma^{2S}\notin\mathfrak A$ and step sizes $h$ such that $t_\ast=\Theta(m_2^{-1/2})$.
\end{itemize}
\end{proposition}

The implications of the proposition on accelerated convergence in NUTS is discussed following \Cref{thm:NUTSmixing}.

By studying the limit $d_1,d_2\to\infty$, we can strengthen the dichotomy.
More precisely, take a two-scale Gaussian distribution and let $d_1,d_2\to\infty$ while keeping $d_2/d_1$, $m_1$ and $m_2$ fixed.
As the deviations $\delta^{2S}_0$ from $f^{2S}_{\unif}$ vanish relative to its scale, \eqref{eq:prop_T} asymptotically only excludes $\mathfrak T$ from intersecting the fixed roots of $f^{2S}_{\unif}$, which corresponds to excluding a discrete set of step sizes with Lebesgue measure zero.
Therefore, if $\gamma^{2S}\in\mathfrak A$, then $t_\ast=\Omega(m_1^{-1/2})$ asymptotically for almost all $h$.
On the other hand, if $\gamma^{2S}\notin\mathfrak A$,
\[ J\ =\ \bigr\{t\in m_2^{-1/2}(0,2\pi)\ :\ f^{2S}_{\unif}(t)<0\bigr\} \]
is a non-empty open interval and $t_\ast=\Theta(m_2^{-1/2})$ if and only if $\mathfrak T\cap J\ne\emptyset$.
This is precisely the case for step sizes in the set $\mathfrak h=\bigcup_{k>0}\frac{J}{2^k-1}$.
In particular, there is a set of step sizes with positive Lebesgue measure for which $t_\ast=\Theta(m_2^{-1/2})$.

In summary, the following dichotomy holds asymptotically:
\begin{itemize}
\item For all $\gamma^{2S}\in\mathfrak A$ and almost all step sizes $h>0$, it holds $t_\ast=\Omega(m_1^{-1/2})$.
\item For all $\gamma^{2S}\notin\mathfrak A$ there exists a set of step sizes $h$ of positive measure such that $t_\ast=\Theta(m_2^{-1/2})$.
\end{itemize}

\begin{remark}\label{rem:technicalissue}
Note that in the accelerated regime, we only assert a lower bound of critical order on the orbit length $t_\ast$.
This stems from the following consideration:
The sign of $f^{2S}_{\unif}$ in the vicinity of the roots of the slower oscillating sine term is primarily determined by the faster oscillating term, see \Cref{fig:sin2S}.
Due to periodicity, the orbit lengths $\mathfrak T$ checked for U-turns in orbit selection might all fall within close proximity of these roots.
In this case, the faster oscillating term may ``hide'' U-turns causing positive signs of $f^{2S}_{\unif}$ across $\mathfrak T$.
Then, orbit selection produces an artificially long orbit beyond criticality.
While this issue does occur in NUTS with fixed step size, it requires close alignment of $\mathfrak T$ with the roots of the slower oscillating term.\footnote{This observation is similar to the looping phenomenon described in \cite{BouRabeeOberdoersterNUTS1}.}
Therefore, randomization of the step size appears to be a promising solution.
\end{remark}

In the preceding considerations, we largely restricted to NUTS using Hamiltonian flow.
The reason being that for NUTS using leapfrog flow, the sign of $f^{2S}_{\unif,h}$ cannot be described by a function depending only on $\kappa$ and $d_2/d_1$ due to the additional correction terms $\beta$, see \eqref{eq:f}.
Therefore, the case of Hamiltonian flow allows for a clearer and more insightful discussion.
However, it remains possible to define the accelerated phase for NUTS using leapfrog flow.

\subsection{The Harmonic Chain}\label{sec:harmonicchain}

Denote the eigenvalues of $\C$ by $\sigma_i^2$.
\Cref{thm:U-turn_concentration} asserts concentration of the U-turn property around the uniform term being of order $\tr(\C^{1/2}) = \sum_{i=1}^d\sigma_i$ with deviations of order $(\tr\C)^{1/2} = \bigr(\sum_{i=1}^d\sigma_i^2\bigr)^{1/2}$, see \eqref{eq:f_deviation}.
In \textsection\ref{sec:iso}, we saw how this yields tight concentration for isotropic Gaussian distributions---already in moderate dimension.
For example in case of the canonical Gaussian, the uniform term is of order $d$ and dominates the deviations of order $d^{1/2}$ even for moderate values of $d$, see \Cref{fig:Uturn_conc_stdGaussian}.

However, if the uniform term is not qualitatively larger than the deviations, the concentration result seizes to be effective.
For instance, this is the case if the variances decay sufficiently quickly as in the following example:
Consider the harmonic chain \cite{krauth2024} of $d$ particles $x_1,\dots,x_d$ in a fixed interval with periodic boundary conditions where each $x_i$ is connected to $x_{i-1}$ and $x_{i+1}$ via springs of fixed stiffness.
The resulting quadratic potential energy yields a Gaussian distribution over the configurations.
The eigenvalues of its covariance matrix decay roughly quadratically towards zero.
Therefore, let $\sigma_i^2=i^{-2}$.
Then, the uniform term is comparable to the harmonic number
\[ \tr(\C^{1/2})\ =\ \sum_{i=1}^di^{-1}\ =\ H_d\ =\ \Theta(\log d) \]
while the deviations are of constant order since the variances are summable.
A pronounced difference in magnitude hence only arises in much higher dimension than in case of the canonical Gaussian.

\section{Mixing of NUTS in Two-scale Gaussian Distributions}\label{sec:mixing}

In this section, we present mixing guarantees for NUTS based on concentration of the U-turn property in two-scale Gaussian distributions, see \textsection\ref{sec:2S}.
In \Cref{prop:2S}, we have established conditions under which the orbit selection simplifies to
\[ I\ \sim\ \mathcal{O}_{x,v}\ =\ \Unif(\mathfrak I_0(t_\ast))\quad\text{for all $(x,v)\in D^{2S}_{\alpha}\times E^{2S}_{\alpha,r}$} \]
with $t_\ast$ given in \eqref{eq:tast2S}.
In case of NUTS using Hamiltonian flow, the index selection is also uniform, $\iota\sim\Unif(I)$.
Together, this yields an integration time $h\iota$ with distribution
\begin{equation}\label{eq:tau_uniformHMC}
\tau_\ast\ =\ \sum_{j\in\mathbb Z}\frac{\max(2^{k_\ast}-|j|,\,0)}{2^{2k_\ast}}\,\delta_{hj}
\end{equation}
where $k_\ast\in\mathbb N$ is such that $t_\ast=h(2^{k_\ast}-1)$.

Consequently, for $x\in D^{2S}_{\alpha}$, NUTS using Hamiltonian flow coincides with the HMC method featuring the integration time distribution $\tau_\ast$ in the event $\{v\in E^{2S}_{\alpha,r}\}$, cf. \Cref{def:HMC}.
In other words, for all $x\in D^{2S}_{\alpha}$, there exists a coupling of $X_{\NUTS}\sim\pi_{\NUTS}(x,\cdot)$ and $X_{\mathrm{HMC}}\sim\pi_{\mathrm{HMC}(\tau_\ast)}(x,\cdot)$ such that $X_{\NUTS}\,\mathbf1_{\{v\in E^{2S}_{\alpha,r}\}}=X_{\mathrm{HMC}}\,\mathbf1_{\{v\in E^{2S}_{\alpha,r}\}}$.
Abusing notation, we write
\begin{equation}\label{eq:NUTSisHMC_exact}
    \pi_{\NUTS}(x,\cdot)\,\mathbf1_{\{v\in E^{2S}_{\alpha,r}\}}\ =\ \pi_{\mathrm{HMC}(\tau_\ast)}(x,\cdot)\,\mathbf1_{\{v\in E^{2S}_{\alpha,r}\}}\quad\text{for all $x\in D^{2S}_{\alpha}$.}
\end{equation}
When using leapfrog flow, the index selection is not uniform but instead follows a Boltzmann-weighted categorical distribution.
In this case, a related statement holds in a more constrained event.

The reduction of NUTS, an HMC method whose integration time distribution $\tau_{x,v}$ depends on position and velocity, to the much simpler instance of HMC with state-independent integration time distribution $\tau_\ast$ enables the transfer of techniques established for HMC to the mixing analysis of NUTS.
In particular, coupling methods accurately capture mixing of the simpler instance of HMC, which can then be transferred to NUTS.
This yields mixing guarantees for NUTS as stated in the next theorem.

To simplify presentation, we introduce notation that omits poly-logarithmic factors in the model parameters.
Let $a=(d_1,d_2,m_1^{-1},m_2,\eps^{-1})\in\mathcal A$ be the domain of all model parameters, where $\eps$ is the desired accuracy in total variation.
For $f,g:\mathcal A\to(0,\infty)$, write $f=\widetilde O(g)$ iff there exists $k\in\mathbb N$ such that
\[ \limsup_{\max a\to\infty}\frac{f(a)}{g(a)\log^k\max a}\ <\ \infty\;. \]
Define $\widetilde\Omega$ and $\widetilde\Theta$ correspondingly.

\begin{theorem}\label{thm:NUTSmixing}
Consider the two-scale Gaussian distribution $\mu=\gamma^{2S}$ defined in \eqref{eq:2S}.
Let $\mathfrak T$, $t_\ast$ and $\mathfrak A$ be as in \eqref{eq:frakT}, \eqref{eq:tast2S} and \eqref{eq:BP}.
Let $\eps>0$ and $x^0\in D^{2S}_{\alpha^0}$ for $\alpha^0=(\alpha^0_1,\alpha^0_2)$ such that $\alpha_i^0=\widetilde O(d_i^{1/2})$, $i=1,2$.
Suppose the assumptions of \Cref{prop:2S} hold.
There exists
\begin{equation}\label{eq:hbar}
    \bar h\ =\ \widetilde\Omega\bigr(m_2^{-1/2}d^{-1/4}\min(m_1^{1/2}t_\ast,\,1)^2\bigr)\quad\text{where $d=d_1+d_2$}
\end{equation}
such that additionally assuming $h\leq\bar h$ in case of NUTS using leapfrog flow, the total variation mixing time of NUTS initialized in $x^0$ to accuracy $\eps$ satisfies
\begin{equation*}\label{eq:thm_mix}
    \tmix(\eps,x^0)\ =\ \inf\bigr\{ n \in \mathbb{N}\,:\,\TV\bigr(\pi_{\NUTS}^n(x^0,\cdot),\,\gamma^{2S}\bigr)\leq \eps \bigr\}\ =\ \widetilde O\bigr(\min(m_1^{1/2}t_\ast,\,1)^{-2}\bigr)\;.
\end{equation*}
In particular, for NUTS using Hamiltonian flow, given that $\max\mathfrak T=\Omega(m_1^{-1/2})$, we have the following dichotomy:
\begin{itemize}
\item For all $\gamma^{2S}\in\mathfrak A$ and all permitted step sizes $h$, it holds $\tmix(\eps,x^0)=\widetilde O(1)$.
\item There exist $\gamma^{2S}\notin\mathfrak A$ and step sizes $h$ such that $\tmix(\eps,x^0)=\widetilde O(\kappa)$.
\end{itemize}
\end{theorem}

This dichotomy follows immediately from the mixing time bound and the corresponding dichotomy for the orbit length selected by NUTS, which is explained in detail in the discussion around \Cref{prop:2S}.

Let us review what the theorem entails for the ability of NUTS to achieve accelerated convergence---the central motif of this study.
We first focus on NUTS using Hamiltonian flow.

On one hand, the theorem asserts that for $\gamma^{2S}\in\mathfrak A$, NUTS mixes in $\widetilde O(1)$ transitions for all permitted step sizes, a term also explained around \Cref{prop:2S}.
Assuming the issue of ``hidden'' U-turns discussed in \Cref{rem:technicalissue} not no occur, NUTS selects orbit of physical time length $t_\ast=\Theta(m_1^{-1/2})$ in this case.
Thus, NUTS requires a total integration time along Hamiltonian flow of $\widetilde O(m_1^{-1/2})$ to mix.
This precisely corresponds to the accelerated relaxation time of critical Randomized HMC described in the introduction.
Therefore, in this case, NUTS achieves the diffusive-to-ballistic speed-up.

On the other hand, the theorem guarantees the existence of two-scale Gaussians and step sizes for which NUTS mixes in $\widetilde O(\kappa)$ transitions.
We expect this bound to be sharp.
In light of the integration times of order $m_2^{-1/2}$ realized by NUTS in these cases, this would be consistent with known convergence lower bounds for HMC with short integration times \cite{LSTlower}.

Note that to achieve acceleration, it likely suffices for NUTS to select long orbits of length $\Omega(m_1^{-1/2})$ in a sufficiently large fraction of transitions.
Even for two-scale Gaussians outside of $\mathfrak A$, a portion of step sizes might yield such orbits.\footnote{The step sizes in $\mathfrak h^c$ in the asymptotic study following \Cref{prop:2S}.}
Therefore, randomizing the step size of NUTS may reliably yield accelerated convergence beyond $\mathfrak A$.
Randomizing the step size in NUTS, primarily with regard to discretization error, is an active area of research, see \cite{BouRabeeCarpenterMarsden2024,BouRabeeCarpenterKleppeMarsden2024,bourabee2025withinorbitadaptiveleapfrognouturn}.

Although not carried out explicitly, similar statements holds for NUTS using leapfrog flow with a suitably adapted definition of accelerated phase $\mathfrak A$.
Interestingly, for $t_\ast=\Omega(m_1^{-1/2})$, the step size upper bound \eqref{eq:hbar} coming out to $\widetilde\Omega(m_2^{-1/2}d^{-1/4})$ is anticipated to be optimal for controlling leapfrog discretization errors in Gaussian distributions, see \cite{chen2023,BouRabeeOberdoersterNUTS1,apers2022,BePiRoSaSt2013}.
For $h$ of order $\bar h$, the computational complexity in terms of number of gradient evaluations of the potential then is $\widetilde O(\kappa^{1/2}d^{1/4})$.
This is consistent with results on HMC with long integration times \cite{apers2022}.

The remainder of this section is dedicated to proving the theorem and organized as follows:
In \textsection\ref{sec:ARchains}, we state the general coupling framework on which the mixing analysis of NUTS is based.
In \textsection\ref{sec:red}, we make precise the reduction of NUTS both using Hamiltonian and leapfrog flow to HMC with integration time distribution $\tau_\ast$.
In \textsection\ref{sec:contr} and \textsection\ref{sec:OS}, we establish Wasserstein contraction and total variation to Wasserstein regularization of this HMC method.
Finally, in \textsection\ref{sec:mixing_proof}, we combine these results to prove \Cref{thm:NUTSmixing}.

\subsection{Mixing of Accept/reject Markov Chains}\label{sec:ARchains}

Accept/reject Markov chains combine two Markov kernels through an accept/reject mechanism.
Let $(\Omega,\mathcal A,\mathbb P)$ be a probability space, $S$ a Polish state space with metric $\mathsf{d}$ and Borel $\sigma$-algebra $\mathcal B$, and $\mathcal P(S)$ the set of probability measures on $(S,\mathcal B)$.
For  any $x\in S$, the transition steps $X^{\acrj}\sim\pi_{\acrj}(x,\cdot)$ of an accept/reject Markov chain take the general form 
\begin{equation} \label{eq:X}
X^{\acrj}(\omega)\ =\ \Phi^{\ac}(\omega,x)\,\ind_{A(x)}(\omega)\ +\ \Phi^{\rj}(\omega,x)\,\ind_{A(x)^c}(\omega)\;, \end{equation}
where $\omega\in\Omega$, $\Phi^{\ac},\Phi^{\rj}:\Omega\times S\to S$ are product measurable and such that $\Phi^{\ac}(\cdot,x)\sim\pi_{\ac}(x,\cdot)$ and $\Phi^{\rj}(\cdot,x)\sim\pi_{\rj}(x,\cdot)$, and the function $A:S\to\mathcal A$ is measurable.  $A(x)$ represents the event of acceptance in the accept/reject mechanism.  In the event of acceptance, the accept/reject chain follows the accept kernel $\pi_{\ac}$;  otherwise, it follows the reject kernel $\pi_{\rj}$.  For convenience, we write
\begin{equation}\label{eq:ar}
    \pi_{\acrj}(x,\cdot)\,\ind_{A(x)}\ =\ \pi_{\ac}(x,\cdot)\,\ind_{A(x)}\;.
\end{equation}
Between two probability measures $\nu,\eta \in \mathcal{P}(S)$,
the total variation distance and the $L^1$-Wasserstein distance with respect to $\mathsf{d}$ are defined as
\begin{equation} \label{eq:tv_coupling}
\TV\bigr(\nu, \eta \bigr) \ = \ \inf\,\mathbb P[X\neq Y]\quad\text{and}\quad\mathcal W^1_{\mathsf{d}}(\nu,\eta) \ = \ \inf\,\mathbb E\,\mathsf{d}(X,Y)
\end{equation}  
with infima taken over all $\mathrm{Law}(X, Y) \in \mathrm{Couplings}(\nu,\eta)$.

\begin{theorem}[\cite{BouRabeeOberdoerster2024,BouRabeeOberdoersterNUTS1}]\label{thm:ARmix}
    Let $\eps>0$ be the desired accuracy, $\nu\in\mathcal P(S)$ the initial distribution, $D\subseteq S$, and $\mu\in\mathcal P(S)$ the invariant measure of the accept/reject kernel $\pi_{\acrj}$.

    Regarding the accept kernel $\pi_{\ac}$, we assume:
    \begin{itemize}
    \item[(i)]
    There exists $\rho>0$ such that for all $x,\tilde{x}\in D$
    \[ \mathcal W^1_{\mathsf{d}}(\pi_{\ac}(x,\cdot),\pi_{\ac}(\tilde x,\cdot))\ \leq\ (1-\rho)\,\mathsf{d}(x,\tilde x)\;. \]
    
    \item[(ii)]
    There exist $C_{Reg},c>0$ such that for all $x,\tilde{x}\in D$
    \[ \TV\bigr(\pi_{\ac}(x,\cdot),\,\pi_{\ac}(\tilde x,\cdot)\bigr)\ \leq\ C_{Reg}\,\mathsf{d}(x,\tilde x)\ +\ c\;. \]
    \end{itemize}

    Regarding the probability of rejection, we assume:
    \begin{itemize}
    \item[(iii)]
    There exists an epoch length $\mathfrak E\in\mathbb N$ and $b>0$ such that
    \[ 2\,\mathfrak E\,\sup\nolimits_{x\in D}\mathbb{P}(A(x)^c)\ +\ C_{Reg}\,\diam_{\mathsf{d}}(D)\,\exp\bigr(-\rho(\mathfrak E-1)\bigr)\ +\ b\ \leq\ 1\ -\ c\;. \]
    \end{itemize}

    Regarding the exit probability of the accept/reject chain from $D$, we assume:
    \begin{itemize}
    \item[(iv)]
    Over the total number of transition steps $\mathfrak H=\mathfrak E\,\lceil b^{-1}\log(2/\eps)\rceil$, it holds that
    \[ \mathbb P\bigr(T\:\leq\:\mathfrak H\bigr)\ \leq\ \eps/4 \]
     for both $X^{\acrj}_0\sim\nu$ and $X^{\acrj}_0\sim\mu$, where $T=\inf\{k\geq0\,:\,X^{\acrj}_k\notin D\}$.
    \end{itemize}
    Then, the mixing time of the accept/reject chain satisfies:
    \[ \tmix(\eps,\nu)\ =\ \inf\bigr\{n\geq0\,:\,\TV(\nu\pi_{\acrj}^n,\,\mu)\leq\eps\bigr\}\ \leq\ \mathfrak H\;. \]
\end{theorem}

Let us sketch how the four assumptions yield a mixing time upper bound.
Via the coupling characterizations of Wasserstein and total variation distances, the contraction of Assumption (i) reduces the distance between two coupled copies of the accept chain, while the partial total variation to Wasserstein regularization of Assumption (ii) allows for exact meeting with a certain probability once the copies are sufficiently close.
Specifically, over an epoch of $\mathfrak E-1$ contractive transitions followed by the regularization, this implies the minorization condition
\[ \TV\bigr(\pi_{\ac}^{\mathfrak E}(x,\cdot),\,\pi_{\ac}^{\mathfrak E}(\tilde x,\cdot)\bigr)\ \leq\ C_{\mathrm{Reg}}\,\diam_{\mathsf{d}}(D)\,e^{-\rho(\mathfrak E-1)}\ +\ c \]
within the domain $D$.  If $c<1$, the right-hand side of this inequality is strictly less than $1$ for sufficiently large $\mathfrak E$.  This ensures non-zero probability of exact meeting between the coupled copies of the accept chain over the epoch.

By isolating the probability of encountering a rejection during the epoch, two copies of the accept/reject chain can be reduced to copies of the accept chain, resulting in
\begin{align*}
    \TV\bigr(\pi_{\acrj}^{\mathfrak E}(x,\cdot),\,\pi_{\acrj}^{\mathfrak E}(\tilde x,\cdot)\bigr)\ &\leq\ 2\,\mathfrak E\,\sup\nolimits_{x\in D}\mathbb{P}(A(x)^c)\ +\ \TV\bigr(\pi_{\ac}^{\mathfrak E}(x,\cdot),\,\pi_{\ac}^{\mathfrak E}(\tilde x,\cdot)\bigr) \\
    &\leq\ 2\,\mathfrak E\,\sup\nolimits_{x\in D}\mathbb{P}(A(x)^c)\ +\ C_{Reg}\,\diam_{\mathsf{d}}(D)\,e^{-\rho(\mathfrak E-1)}\ +\ c\;.
\end{align*}
Assumption (iii) ensures that the probability of rejection over the epochs $\mathfrak E$ is suitably controlled for the right hand side to be bounded above by $1-b\leq e^{-b}$.  By iterating over $\lceil b^{-1}\log(2/\eps)\rceil$ epochs, mixing to accuracy $\eps$ is achieved.  Assumption (iv) allows to restrict the argument to the domain $D$.

\subsection{Reduction of NUTS to HMC with State-independent Integration Times}\label{sec:red}

For $U\sim\Unif([0,1])$, $(x,v) \in \mathbb{R}^{2d}$ and an index orbit $I \subset \mathbb{Z}$, define the event
\begin{equation}\label{eq:AI}
    A_I(x,v)\ =\ \Bigr\{U\ \leq\ |I|\min_{i\in I}e^{-(H\circ\varphi_{hi}-H)(x,v)}\Bigr(\sum_{i\in I}e^{-(H\circ\varphi_{hi}-H)(x,v)}\Bigr)^{-1}\Bigr\}\;.
\end{equation}

\begin{lemma}\label{lem:MultHMCtoUnifHMC}
Suppose the assumptions of \Cref{prop:2S} hold.
Let $\tau_\ast$ be as in \eqref{eq:tau_uniformHMC} with $t_\ast$ as in \eqref{eq:tast2S}.
Consider NUTS using
\begin{itemize}
\item[(i)] Hamiltonian flow, or
\item[(ii)] leapfrog flow with step-size $h>0$ such that $h^2m_2\leq1$.
\end{itemize}
In the event $\{v\in E^{2S}_{\alpha,r}\}\cap A_I(x,v)$ where $v$ is the initial velocity and $I$ the selected index orbit, NUTS coincides with $\mathrm{HMC}(\tau_\ast)$ as in \Cref{def:HMC},
\begin{equation}\label{eq:MultHMCtoUnifHMC}
    \pi_{\NUTS}(x,\cdot)\,\ind_{\{v\in E^{2S}_{\alpha,r}\}\cap A_I(x,v)}\ =\ \pi_{\mathrm{HMC}(\tau_\ast)}(x,\cdot)\,\ind_{\{v\in E^{2S}_{\alpha,r}\}\cap A_I(x,v)}\quad\text{for all $x\in D^{2S}_{\alpha}$.}
\end{equation}
Further, for all $x\in D^{2S}_{\alpha}$, it holds
\begin{align}\label{eq:rej_prob}
    \mathbb P\bigr(\{v\in E^{2S}_{\alpha,r}\}\cap A_I(x,v)\bigr)
    \ \geq\ 1&-8\,e^{-\frac18\min_{i\in\{1,2\}}r_i^2/d_i}\nonumber\\
    &-2\hbar^2\max_{i\in\{1,2\}}\bigr(m_i\max(\alpha_i,r_i)+\hbar^2m_i^2d_i\bigr)
\end{align}
with $\hbar=0$ in case of Hamiltonian flow, and $\hbar=h$ in case of leapfrog flow.
\end{lemma}

In case of NUTS using Hamiltonian flow, $A_I(x,v)=\Omega$ so that \eqref{eq:MultHMCtoUnifHMC} recovers \eqref{eq:NUTSisHMC_exact}.

\begin{proof}[Proof of \Cref{lem:MultHMCtoUnifHMC}]
Let $x\in D^{2S}_{\alpha}$.
We first prove \eqref{eq:MultHMCtoUnifHMC} stating that NUTS coincides with HMC with integration time distribution $\tau_\ast$ in the event $\{v\in E^{2S}_{\alpha,r}\}\cap A_I(x,v)$.
As discussed at the beginning of \textsection\ref{sec:mixing}, NUTS coincides with this instance of HMC in case of uniform orbit selection $I\sim\Unif(\mathfrak I_0(t_\ast))$ and uniform index selection $\iota\sim\Unif(I)$.

For NUTS using Hamiltonian flow, the former is ensured in the event $\{v\in E^{2S}_{\alpha,r}\}$ by \Cref{prop:2S}, while the latter holds due to the fact that Hamiltonian flow preserves the Hamiltonian.

For NUTS using leapfrog flow, \Cref{prop:2S} also guarantees uniform orbit selection in the event $\{v\in E^{2S}_{\alpha,r}\}$.
However, due to discretization error, the Boltzmann-weighted index selection is not uniform.
Instead, we restrict to the event $A_I(x,v)$ to obtain uniformity.
Indeed, given $(x,v) \in \mathbb{R}^{2d}$ and an index orbit $I \subset \mathbb{Z}$, the index $\iota\sim\cat\bigr(e^{-(H\circ\varphi_{hi}-H)(x,v)}\bigr)_{i\in I}$ in NUTS can be expressed as
\[ \iota\ =\ \iota_{\ac}\,\ind_{A_I(x,v)}+\iota_{\rj}\,\ind_{A_I(x,v)^c} \]
where
\[ \iota_{\ac}\sim\Unif(I)\quad\text{and}\quad \iota_{\rj}\sim\cat\bigr(e^{-(H\circ\Phi_h^i-H)(x,v)}-\min\nolimits_{i\in I}e^{-(H\circ\Phi_h^i-H)(x,v)}\bigr)_{i\in I} \]
are independent.
This dissects the Boltzmann-weighted categorical distribution into its maximal uniform part and a categorical remainder, cf. \cite{BouRabeeOberdoersterNUTS1}.
Consequently, in the event $\{v\in E^{2S}_{\alpha,r}\}\cap A_I(x,v)$ both orbit and index selection are uniform, implying that NUTS coincides with the desired instance of HMC.

For \eqref{eq:rej_prob} with Hamiltonian flow, note that since $A_I(x,v)=\Omega$ in this case,
\begin{equation}\label{eq:E2Sbound}
    \mathbb P\bigr(\{v\in E^{2S}_{\alpha,r}\}\cap A_I(x,v)\bigr)\ =\ \mathbb P\bigr(v\in E^{2S}_{\alpha,r}\bigr)\ =\ \gamma_d\bigr(E^{2S}_{\alpha,r}\bigr)
    \ \geq\ 1-8\,e^{-\frac18\min_{i\in\{1,2\}}r_i^2/d_i}
\end{equation}
by \eqref{eq:Eprob_2S}.
For \eqref{eq:rej_prob} concerning leapfrog flow, observe
\begin{equation}\label{eq:rej_prob_sum}
    \mathbb P\bigr(\{v\in E^{2S}_{\alpha,r}\}\cap A_I(x,v)\bigr)\ =\ \mathbb P\bigr(v\in E^{2S}_{\alpha,r}\bigr)-\mathbb P\bigr(\{v\in E^{2S}_{\alpha,r}\}\cap A_I(x,v)^c\bigr)\;.
\end{equation}
For the second term, the definition of $A_I(x,v)$ yields with the Hamiltonian \eqref{eq:H2S} and the leapfrog flow \eqref{eq:twoscaleLF} in the two-scale distribution
\begin{align*}
    &\mathbb P\bigr(A_I(x,v)^c|v,I\bigr)
\ =\ 1-|I|\min\nolimits_{l\in I}e^{-(H_{2S}\circ\Phi_{h,2S}^l-H_{2S})(x,v)}\Bigr(\sum_{l\in I}e^{-(H_{2S}\circ\Phi_{h,2S}^l-H_{2S})(x,v)}\Bigr)^{-1} \\
&\qquad\leq\ 1-e^{-2\sup\nolimits_{l\in I}|H_{2S}\circ\Phi_{h,2S}^l-H_{2S}|(x,v)}
\ \leq\ 2\sup\nolimits_{l\in\mathbb Z}\bigr|H_{2S}\circ\Phi_{h,2S}^l-H_{2S}\bigr|(x,v) \;.
\end{align*}
Inserting this bound into the second term of \eqref{eq:rej_prob_sum} shows
\begin{align*}
\mathbb P\bigr(\{v\in E^{2S}_{\alpha,r}\}\cap A_I(x,v)^c\bigr)\ &=\ \mathbb E\bigr(\mathbb P\bigr(A_I(x,v)^c|v,I\bigr)\,\ind_{v\in E^{2S}_{\alpha,r}}\bigr)\\
&\leq\ 2\,\mathbb E\bigr(\sup\nolimits_{l\in\mathbb Z}\bigr|H_{2S}\circ\Phi_{h,2S}^l-H_{2S}\bigr|(x,v)\,\ind_{v\in E^{2S}_{\alpha,r}}\bigr)\;. \end{align*}
Together with \eqref{eq:E2Sbound}, it holds
\begin{align}
    \mathbb P\bigr(\{v\in E^{2S}_{\alpha,r}\}\cap A_I(x,v)\bigr)\ \geq\ 1&-8\,e^{-\frac18\min_{i\in\{1,2\}}r_i^2/d_i}\nonumber\\
    &-2\,\mathbb E\bigr(\sup\nolimits_{l\in\mathbb Z}\bigr|H_{2S}\circ\Phi_{h,2S}^l-H_{2S}\bigr|(x,v)\,\ind_{v\in E^{2S}_{\alpha,r}}\bigr)\;.
\end{align}
Therefore, \eqref{eq:rej_prob} follows from
\[ \sup\nolimits_{l\in\mathbb Z}\bigr|H_{2S}\circ\Phi_{h,2S}^l-H_{2S}\bigr|(x,v)\ \leq\ h^2\max_{i\in\{1,2\}}\bigr(m_i\max(\alpha_i,r_i)+h^2m_i^2d_i\bigr)\quad\text{for all $v\in E^{2S}_{\alpha,r}$} \]
which we show next.

The two-scale energy error decomposes into the energy errors at the individual scales
\begin{equation}\label{eq:2SdeltaHdec}
    \bigr(H_{2S}\circ\Phi_{h,2S}^l-H_{2S}\bigr)(x,v)\ =\ \sum_{i=1}^2\bigr(H_{m_i}\circ\Phi_{h,m_i}^l-H_{m_i}\bigr)(x^i,v^i)\quad\text{for all $l\in\mathbb Z$.}
\end{equation}
We estimate the energy error at each scale separately.  Therefore, let
\[ \bigr(\mathrm x, \mathrm v, \mathrm d, \mathrm m, \mathrm a, \mathrm r\bigr)\ \in\ \bigr\{\bigr(x^i, v^i, d_i, m_i, \alpha_i, r_i\bigr)\bigr\}_{i=1,2}\;. \]
It then holds
\[ \bigr(H_{\mathrm m}\circ\Phi_{h,\mathrm m}^l-H_{\mathrm m}\bigr)(\mathrm x,\mathrm v)\ =\ \frac{h^2\mathrm m}{8}\Bigr(\bigr|\mathrm m^{1/2}\proj_1^{\mathrm d}\Phi_{h,\mathrm m}^l(\mathrm x,\mathrm v)\bigr|^2-|\mathrm m^{1/2}\mathrm x|^2\Bigr) \]
which follows from the fact that $\Phi_{h,\mathrm m}$ preserves the modified Hamiltonian $H_{h,\mathrm m}(\mathrm x,\mathrm v)=H_{\mathrm m}(\mathrm x,\mathrm v)-\frac{h^2\mathrm m}{8}|\mathrm m^{1/2}\mathrm x|^2$.  Taking absolute values and applying the triangle inequality yields
\begin{equation}\label{eq:deltaHabs}
\bigr|H_{\mathrm m}\circ\Phi_{h,\mathrm m}^l-H_{\mathrm m}\bigr|(\mathrm x,\mathrm v)\ \leq\ \frac{h^2\mathrm m}{8}\Bigr(\Bigr|\bigr|\mathrm m^{1/2}\proj_1^{\mathrm d}\Phi_{h,\mathrm m}^l(\mathrm x,\mathrm v)\bigr|^2-\mathrm d\Bigr|\ +\ \bigr||\mathrm m^{1/2}\mathrm x|^2-\mathrm d\bigr|\Bigr)\;.
\end{equation}
The second term satisfies
\begin{equation}\label{eq:ISdalpha}
\bigr||\mathrm m^{1/2}\mathrm x|^2-\mathrm d\bigr|\ \leq\ \mathrm a
\end{equation}
since $\mathrm x\in\mathrm m^{-1/2} D_{\mathrm a}^{\mathrm d}$.
For the first term, inserting the explicit leapfrog flow \eqref{eq:mLF}, using \eqref{eq:ISdalpha},
\[ \max\Bigr(\bigr||\mathrm v|^2-\mathrm d\bigr|,\,\sup\nolimits_{\mathrm x\in \mathrm m^{-1/2} D_{\mathrm a}^{\mathrm d}}|\mathrm m^{1/2}\mathrm x\cdot \mathrm v|\Bigr)\ \leq\ \mathrm r \]
valid since $\mathrm v\in E^{\mathrm d}_{\mathrm a,\mathrm r}$, $h^2\mathrm m\leq1$, $\mathrm r\leq\mathrm d$, and applying the elementary bound $1+\mathsf{b}^2/3\geq(1-\mathsf{b}^2/4)^{-1}$ for $\mathsf{b}\in(0,1]$, we obtain for all $t\in h\mathbb Z$
\begin{align}\label{eq:Phi2-d}
	&\Bigr|\bigr|\mathrm m^{1/2}\proj_1^{\mathrm d}\Phi_{h,\mathrm m}^{t/h}(\mathrm x,\mathrm v)\bigr|^2-\mathrm d\Bigr|\nonumber\\
    &\quad\leq\ \cos_{h,\mathrm m}^2(\mathrm m^{1/2}t)\,\bigr||\mathrm m^{1/2}\mathrm x|^2-\mathrm d\bigr|\ +\ \sin_{h,\mathrm m}^2(\mathrm m^{1/2}t)(1-h^2\mathrm m/4)^{-1}\bigr||\mathrm v|^2-\mathrm d\bigr|\nonumber\\
	&\quad\quad\ +\sin_{h,\mathrm m}^2(\mathrm m^{1/2}t)\bigr((1-h^2\mathrm m/4)^{-1}-1\bigr)\mathrm d\nonumber\\
    &\quad\quad\ +|\sin_{h,\mathrm m}|(2\mathrm m^{1/2}t)(1-h^2\mathrm m/4)^{-1/2}|\mathrm m^{1/2}\mathrm x\cdot \mathrm v|\nonumber\\
	&\quad\leq\ \cos_{h,\mathrm m}^2(\mathrm m^{1/2}t)\mathrm a+\sin_{h,\mathrm m}^2(\mathrm m^{1/2}t)(1+h^2\mathrm m/3)\mathrm r+h^2\mathrm m\mathrm d/3+(1+h^2\mathrm m/3)\mathrm r \nonumber\\
	&\quad\leq\ \max(\mathrm a,\mathrm r)+\mathrm r+h^2\mathrm m\mathrm d\;.
\end{align}
Inserting this together with \eqref{eq:ISdalpha} into \eqref{eq:deltaHabs} yields, for all $l\in\mathbb Z$,
\[ \bigr|H_{\mathrm m}\circ\Phi_{h,\mathrm m}^l-H_{\mathrm m}\bigr|(\mathrm x,\mathrm v)\ \leq\ \frac{h^2\mathrm m}{8}\bigr(3\max(\mathrm a,\mathrm r)+h^2\mathrm m\mathrm d\bigr)\ \leq\ \frac12h^2\mathrm m\max(\mathrm a,\mathrm r)+\frac18h^4\mathrm m^2\mathrm d\;. \]
Combining these bounds for the energy error at the individual scales to a bound for the the two-scale energy error shows
\begin{align*}
    \sup\nolimits_{l\in\mathbb Z}\bigr|H_{2S}\circ\Phi_{h,2S}^l-H_{2S}\bigr|(x,v)\ &\leq\ \sum_{i=1}^2\bigr(\frac12h^2m_i\max(\alpha_i,r_i)+\frac18h^4m_i^2d_i\bigr)\\
    &\leq\ h^2\max_{i\in\{1,2\}}\bigr(m_i\max(\alpha_i,r_i)+h^2m_i^2d_i\bigr)\;,
\end{align*}
closing the proof.
\end{proof}

The reduction of NUTS to HMC with state-independent integration time distribution requires $x\in D^{2S}_\alpha$.
This is a natural assumption as the target distribution $\gamma^{2S}$ concentrates in those sets, see \eqref{eq:D2Sbound}.
Conveniently, a bound from the last proof allows us to control the exit probability from such sets.
Let $T^{2S}_\alpha$ be the first exit time from $D^{2S}_\alpha$, i.e.,
\[ T^{2S}_\alpha\ =\ \min\{k\geq0\,:\,X_k\notin D^{2S}_\alpha\}\;. \]

\begin{lemma}\label{lem:stab_leapfrog}
Consider $\mu=\gamma^{2S}$ as defined in \eqref{eq:2S} with $0<m_1\leq m_2$ and $d_1,d_2\in\mathbb N$.
Let $\alpha^0=(\alpha^0_1,\alpha^0_2)$, $r=(r_1,r_2)$ and set
\begin{equation}\label{eq:alphan}
    \alpha_\hbar(n)\ =\ \bigr(\max(\alpha^0_i,r_i)+n(r_i+\hbar^2m_id_i)\bigr)_{i\in\{1,2\}}\quad\text{for $n\in\mathbb N$ and $\hbar\geq0$.}
\end{equation}
Assume $\alpha(n-1)_1\leq d_1$ and $\alpha(n-1)_2\leq d_2$.
Started in any $x^0\in D^{2S}_{\alpha^0}$, NUTS using Hamiltonian flow or leapfrog flow with step-size $h>0$ such that $h^2m_2\leq 1$ satisfies
\begin{equation}\label{eq:stab_NUTS}
    \mathbb P(T^{2S}_{\alpha_\hbar(n)}\leq n)\ \leq\ 8n\,e^{-\frac18\min_{i\in\{1,2\}}r_i^2/d_i}
\end{equation}
with $\hbar=0$ in case of Hamiltonian flow, and $\hbar=h$ in case of leapfrog flow.
\end{lemma}

\begin{proof}[Proof of \Cref{lem:stab_leapfrog}]
We prove the lemma for NUTS using leapfrog flow while the result for Hamiltonian flow follows analogously.
Let $x\in D^{2S}_{\alpha_h(k)}$ with $\alpha_h(k)$ as defined in \eqref{eq:alphan} such that $\alpha_h(k)_1\leq d_1$ and $\alpha_h(k)_2\leq d_2$, and $v\sim\gamma_{d_1+d_2}$.
Further, let $r=(r_1,r_2)$ with $r_1\leq d_1$ and $r_2\leq d_2$.
By \eqref{eq:Phi2-d},
\begin{equation*}\label{eq:Phi2-d_2}
    \Bigr|\bigr|m_i^{1/2}\proj_1^{d_i}\Phi_{h,m_i}^{l}(x^i,v^i)\bigr|^2-d_i\Bigr|\ \leq\ \max(\alpha_h(k)_i,r_i)+r_i+h^2m_id_i\ =\ \alpha_h(k+1)_i
\end{equation*}
for $v\in E^{2S}_{\alpha_h(k),r}$ and all $l\in\mathbb Z$.
In particular,
\[ \proj_1^{d_1+d_2}\Phi_{h,2S}^l(x,v)\ \in\ D^{2S}_{\alpha_h(k+1)}\quad\text{for all $l\in\mathbb Z$} \]
with probability at least $\mathbb P(E^{2S}_{\alpha_h(k),r})\geq1-8\,e^{-\frac18\min_{i\in\{1,2\}}r_i^2/d_1}$ by \eqref{eq:E2Sbound}.
As any transition of NUTS from $x$ is of the form given on the left hand side, the maximum expansion of $D^{2S}_{\alpha_h(k)}$ by NUTS is confined to $D^{2S}_{\alpha_h(k+1)}$ with high probability.
Starting from $D^{2S}_{\alpha^0}$, this yields \eqref{eq:stab_NUTS}.
\end{proof}

\subsection{Wasserstein Contraction of HMC}\label{sec:contr}

\begin{lemma}\label{lem:contr}
Consider $\gamma^{2S}$ as defined in \eqref{eq:2S} with $d_1,d_2\in\mathbb N$ and $0<m_1\leq m_2$.
Set
\[ \rho_\hbar\ =\ \frac12\inf_{m\in\{m_1,m_2\}}\int\sin^2(\beta_{\hbar^2m}m^{1/2}t)\,\tau_\ast(dt)\quad\text{for $\hbar\geq0$.} \]
HMC with integration time distribution $\tau_\ast$ as defined in \eqref{eq:tau_uniformHMC} satisfies
\[ \mathcal W^1_{|\cdot|_{2S}}\bigr(\pi_{\HMC(\tau_\ast)}(x,\cdot),\,\pi_{\HMC(\tau_\ast)}(\tilde x,\cdot)\bigr)\ \leq\ (1-\rho)|x-\tilde x|_{2S}\quad\text{for all $x,\tilde x\in\mathbb R^d$} \]
with rate $\rho=\rho_0$ in case Hamiltonian flow is used, and $\rho=\rho_h$ in case of leapfrog flow with step-size $h>0$ such that $h^2m_2\leq 1$ is used.
In both cases, there exists an absolute constant $c>0$ such that the rate is lower bounded according to $\rho_\hbar\geq c\min\bigr(m_1^{1/2}t_\ast,\,1\bigr)^2$.
\end{lemma}

An analogous assertion holds for the standard norm.

\begin{proof}
We prove the contraction for leapfrog flow.
The result for Hamiltonian flow follows analogously.
Let $x=(x^1,x^2),\tilde x=(\tilde x^1,\tilde x^2)\in\mathbb R^{d_1+d_2}$ and consider the synchronous coupling
\[ (X,\widetilde X)\ =\ \bigr(\proj_1^{d_1+d_2}\Phi_{h,2S}^{T/h}(x,v),\,\proj_1^{d_1+d_2}\Phi_{h,2S}^{T/h}(\tilde x,v)\bigr) \]
of $\pi_{\HMC(\tau_\ast)}(x,\cdot)$ and $\pi_{\HMC(\tau_\ast)}(\tilde x,\cdot)$ obtained by using the same integration time $T\sim\tau_\ast$ and velocity $v\sim\gamma_{d_1+d_2}$ for both copies.
Inserting \eqref{eq:twoscaleLF} with \eqref{eq:mLF} into \eqref{eq:2Snorm} yields
\begin{align*}
    &\mathcal W^1_{|\cdot|_{2S}}\bigr(\pi_{\HMC(\tau_\ast)}(x,\cdot),\,\pi_{\HMC(\tau_\ast)}(\tilde x,\cdot)\bigr)\ \leq\ \mathbb E|X-\widetilde X|_{2S}\\
    &\quad\leq\ \Bigr(\mathbb E\cos_{h,m_1}^2(m_1^{1/2}T)\,m_1|x^1-\tilde x^1|^2+\mathbb E\cos_{h,m_2}^2(m_2^{1/2}T)\,m_2|x^2-\tilde x^2|^2\Bigr)^{1/2} \\
    &\quad\leq\ \sup_{m\in\{m_1,m_2\}}\Bigr(\int\cos_{h,m}^2(m^{1/2}t)\,\tau_\ast(dt)\Bigr)^{1/2}|x-\tilde x|_{2S} \\
    &\quad\leq\ \Bigr(1-\frac12\inf_{m\in\{m_1,m_2\}}\int\sin_{h,m}^2(m^{1/2}t)\,\tau_\ast(dt)\Bigr)|x-\tilde x|_{2S}
    \ =\ (1-\rho_h)|x-\tilde x|_{2S}\;.
\end{align*}

To lower bound the rate $\rho_\hbar$, let $m\in\{m_1,m_2\}$.
Then,
\begin{equation}\label{eq:contr_rateeq}
    \int\sin_{\hbar,m}^2(m^{1/2}t)\,\tau_\ast(dt)\ =\ 2\int_{(0,\infty)}\sin^2(\beta_{\hbar^2m}m^{1/2}t)\,\tau_\ast(dt)\;.
\end{equation}
Note that
\begin{equation}\label{eq:contr_est}
    \sin(\beta_{\hbar^2m}m^{1/2}t)\ \geq\ \frac2\pi\beta_{\hbar^2m}m^{1/2}t\quad\text{for $\beta_{\hbar^2m}m^{1/2}t\in(0,\pi/2)$.}
\end{equation}
As $t\leq \hbar2^{k_\ast}$ for $t\in\operatorname{supp}(\tau_\ast)$, assuming $\beta_{\hbar^2m}m^{1/2}\hbar2^{k_\ast}\leq\pi/2$ ensures \eqref{eq:contr_est} to hold for all $t\in\operatorname{supp}(\tau_\ast)$ so that
\begin{align*} 2\int_{(0,\infty)}\sin^2(\beta_{\hbar^2m}m^{1/2}t)\,\tau_\ast(dt)\ &\geq\ \frac8{\pi^2}\beta_{\hbar^2m}^2m\int_{(0,\infty)}t^2\,\tau_\ast(dt)\ \geq\ \frac8{\pi^2}m\sum_{j>0}j^2\frac{2^{k_\ast}-j}{2^{2k_\ast}}\\
&=\ \frac2{3\pi^2}\bigr(m^{1/2}\hbar2^{k_\ast}\bigr)^2\ \geq\ \frac2{3\pi^2}\bigr(m^{1/2}t_\ast\bigr)^2\;.
\end{align*}
For $\beta_{\hbar^2m}m^{1/2}\hbar2^{k_\ast}>\pi/2$, it is easy to see that \eqref{eq:contr_rateeq} is bounded below by some absolute constant.
Therefore, there exists an absolute constant $c>0$ such that
\[ \int\sin_{\hbar,m}^2(m^{1/2}t)\,\tau_\ast(dt)\ \geq\ 2c\min\bigr(m^{1/2}t_\ast,\,1\bigr)^2\quad\text{for $m\in\{m_1,m_2\}$.} \]
Since $m_1\leq m_2$, we see
\[ \rho_\hbar\ =\ \frac12\inf_{m\in\{m_1,m_2\}}\Bigr(\int\sin_{\hbar,m}^2(m^{1/2}t)\,\tau_\ast(dt)\Bigr)^{1/2}\ \geq\ c\min\bigr(m_1^{1/2}t_\ast,\,1\bigr)^2\;. \]
\end{proof}

\subsection{Partial Total Variation to Wasserstein Regularization of HMC}\label{sec:OS}

\begin{lemma}\label{lem:OS}
Suppose the assumptions of \Cref{lem:contr} hold.
Let $B\subset\mathbb R$ and set
\[ C_{\mathrm{reg},\hbar}(B)\ =\ \sup_{m\in\{m_1,m_2\}}\Bigr(\int_{B^c}\cot^2(\beta_{\hbar^2m}m^{1/2}t)\,\tau_\ast(dt)\Bigr)^{1/2}\quad\text{for $\hbar\geq0$.} \]
Then, HMC with integration time distribution $\tau_\ast$ as defined in \eqref{eq:tau_uniformHMC} satisfies
\[ \TV\bigr(\pi_{\HMC(\tau_\ast)}(x,\cdot),\,\pi_{\HMC(\tau_\ast)}(\tilde x,\cdot)\bigr)\ \leq\ C_{\mathrm{reg}}\,|x-\tilde x|_{2S}\ +\ \tau_\ast(B)\quad\text{for all $x,\tilde x\in\mathbb R^d$} \]
with regularization constant $C_{\mathrm{reg}}=C_{\mathrm{reg},0}(B)$ in case Hamiltonian flow is used, and $C_{\mathrm{reg}}=C_{\mathrm{reg},h}(B)$ in case of leapfrog flow with step-size $h>0$ such that $h^2m_2\leq 1$ is used.
In both cases, for all $\delta\in(0,1)$ and all $x,\tilde x\in\mathbb R^d$, it holds
\begin{equation}\label{eq:OS}
    \TV\bigr(\pi_{\HMC(\tau_\ast)}(x,\cdot),\,\pi_{\HMC(\tau_\ast)}(\tilde x,\cdot)\bigr)\ \leq\ 2\delta^{-1}|x-\tilde x|_{2S}+24\max\bigr(1,\,(m_1^{1/2}t_\ast)^{-1}\bigr)\,\delta\;.
\end{equation}
\end{lemma}

\begin{proof}
We prove the regularization for HMC using leapfrog flow while the result using Hamiltonian flow again follows analogously.
Let $x=(x^1,x^2),\tilde x=(\tilde x^1,\tilde x^2)\in\mathbb R^{d_1+d_2}$, $T\sim\tau_\ast$ and $v,\tilde v\sim\gamma_{d_1+d_2}$.
Consider the coupling
\[ (X,\widetilde X)\ =\ \bigr(\proj_1^{d_1+d_2}\Phi_{h,2S}^{T/h}(x,v),\,\proj_1^{d_1+d_2}\Phi_{h,2S}^{T/h}(\tilde x,\tilde v)\bigr) \]
of $\pi_{\HMC(\tau_\ast)}(x,\cdot)$ and $\pi_{\HMC(\tau_\ast)}(\tilde x,\cdot)$ using the same path length but distinct velocities $v$ and $\tilde v$ coupled such that with maximal probability
\[\begin{aligned}[t]
    \tilde v^1 &= v^1+s^1(T)\quad\text{with}\quad s^1(T) = \cot_{h,m_1}(m_1^{1/2}T)(1-h^2m_1/4)^{1/2}m_1^{1/2}(x^1-\tilde x^1)\quad\text{and}\\
    \tilde v^2 &= v^2+s^2(T)\quad\text{with}\quad s^2(T) = \cot_{h,m_2}(m_2^{1/2}T)(1-h^2m_2/4)^{1/2}m_2^{1/2}(x^2-\tilde x^2)
\end{aligned}\]
in which case the copies meet, i.e., $X=\widetilde X$.
Writing $s=(s^1,s^2)$, we then have for any $B\subset\mathbb R$,
\begin{align*}
	\TV\bigr(\pi_{\HMC(\tau_\ast)}(x,\cdot),\,\pi_{\HMC(\tau_\ast)}(\tilde x,\cdot)\bigr)
\ &\leq\ \mathbb P(X\ne\widetilde X)
\ \leq\ \int_{B^c}\mathbb P\bigr(\tilde v\ne v+s(t)\bigr)\,\tau_\ast(dt)+\tau_k(B) \\
&=\ \int_{B^c}\TV\bigr(\gamma_{d_1+d_2}+s(t),\gamma_{d_1+d_2}\bigr)\,\tau_\ast(dt)+\tau_\ast(B)\;,
\end{align*}
where we used maximality of the coupling.
By Pinsker's inequality
\begin{align*}
    \TV\bigr(\gamma_{d_1+d_2}+s(t),\gamma_{d_1+d_2}\bigr)\ &\leq\ \sqrt2\, \mathrm{H}\bigr(\gamma_{d_1+d_2}\bigr|\gamma_{d_1+d_2}+s(t)\bigr)^{1/2}\ =\ |s(t)|\\
    &=\ \sqrt{|s^1(t)|^2+|s^2(t)|^2}
\end{align*}
so that together with Jensen's inequality, the definitions of $s^1,s^2$ and the two-scale norm \eqref{eq:2Snorm},
\begin{align*}
    &\TV\bigr(\pi_{\HMC(\tau_\ast)}(x,\cdot),\,\pi_{\HMC(\tau_\ast)}(\tilde x,\cdot)\bigr)\ \leq\ \int_{B^c}\sqrt{|s^1(t)|^2+|s^2(t)|^2}\,\tau_\ast(dt)+\tau_\ast(B)\\
    &\leq \Bigr(\int_{B^c}\bigr(|s^1(t)|^2+|s^2(t)|^2\bigr)\,\tau_\ast(dt)\Bigr)^{1/2} + \tau_\ast(B)\\
    &\leq \Bigr(\int_{B^c}\Bigr(\cot^2_{h,m_1}(m_1^{1/2}t)\,m_1|x^1-\tilde x^1|^2+\cot^2_{h,m_2}(m_2^{1/2}t)\,m_2|x^2-\tilde x^2|^2\Bigr)\,\tau_\ast(dt)\Bigr)^{1/2} \!\!+ \tau_\ast(B) \\
    &\leq \sup_{m\in\{m_1,m_2\}}\Bigr(\int_{B^c}\cot^2_{h,m}(m^{1/2}t)\,\tau_\ast(dt)\Bigr)^{1/2}\bigr(m_1|x^1-\tilde x^1|^2+m_2|x^2-\tilde x^2|^2\bigr)^{1/2} + \tau_\ast(B) \\
    &\leq \sup_{m\in\{m_1,m_2\}}\Bigr(\int_{B^c}\cot^2_{h,m}(m^{1/2}t)\,\tau_\ast(dt)\Bigr)^{1/2}|x-\tilde x|_{2S} + \tau_\ast(B)\;.
\end{align*}
This concludes the proof of the regularization.

\begin{figure}[t]
\centering
\begin{tikzpicture}[scale=1.2]
\clip ({-pi-1},-0.7) rectangle ({2*pi+1},5.5);

\draw[line width=0.7pt,->] ({-pi-1},0) -- ({2*pi+1},0) node [pos=1,below left] {$t$};
\draw[line width=0.7pt,->] (0,-0.1) -- (0,5.5);

\draw [gray,line width=1.5pt,smooth,samples=20,domain=(-2*pi+0.1):(-pi-0.1)] plot({\x},{cot(\x r)^2});
\draw [gray,line width=1.5pt,smooth,samples=20,domain=(-0.414):(-0.1)] plot({\x},{cot(\x r)^2});
\draw [gray,line width=1.5pt,smooth,samples=20,domain=(-pi+0.1):(-0.428)] plot({\x},{cot(\x r)^2});
\draw [gray,line width=1.5pt,smooth,samples=20,domain=0.1:(pi-0.1)] plot({\x},{cot(\x r)^2});
\draw [gray,line width=1.5pt,smooth,samples=20,domain=(pi+0.1):(2*pi-0.1)] plot({\x},{cot(\x r)^2});

\draw[line width=0.7pt,dashed] ({-pi-1},5) -- (-1.0,5);
\draw[line width=0.7pt,dashed] (0,5) -- ({2*pi+1},5);
\draw[line width=0.7pt,dashed] (0.42,0) -- (0.42,5);
\draw[line width=0.7pt,dashed] (-0.42,0) -- (-0.42,4.8);
\draw[line width=0.7pt,dashed] (2.72,0) -- (2.72,5);
\draw[line width=0.7pt,dashed] (-2.72,0) -- (-2.72,5);
\draw[line width=0.7pt,dashed] (3.56,0) -- (3.56,5);
\draw[line width=0.7pt,dashed] (-3.56,0) -- (-3.56,5);
\draw[line width=0.7pt,dashed] (5.86,0) -- (5.86,5);
\draw[line width=0.7pt,dashed] (-5.86,0) -- (-5.86,5);

\draw[line width=2.5pt] (-3.56,0) -- (-2.72,0);
\draw[line width=2.5pt] (-0.42,0) -- (0.42,0);
\draw[line width=2.5pt] (2.72,0) -- (3.56,0);
\draw[line width=2.5pt] (5.86,0) -- (6.70,0);

\foreach \k in {0,...,24} {
\draw ({2*pi*\k/25},0) -- ({2*pi*\k/25},{2*(1-\k/25)});
\filldraw ({2*pi*\k/25},{2*(1-\k/25)}) circle (1pt);
}
\foreach \k in {-20,...,-1} {
\draw ({2*pi*\k/25},0) -- ({2*pi*\k/25},{2*(1+\k/25)});
\filldraw ({2*pi*\k/25},{2*(1+\k/25)}) circle (1pt);
}

\node at (-pi,0) [below] {$B_{m,\delta}^{-1}$};
\node at (0,0) [below] {$B_{m,\delta}^0$};
\node at (pi,0) [below] {$B_{m,\delta}^1$};
\node at (2*pi,0) [below] {$B_{m,\delta}^{l_{max}}$};
\draw[line width=1pt] (-0.1,5) -- (0.1,5);
\node at (0,5) [left] {$\cot^2\delta$};
\node[gray] at (3.7,5.25) [right] {$\cot_{h,m}^2(m^{1/2}t)$};
\node at (pi/2,1.5) [above] {$\tau_\ast$};
\node at (-0.05,2) [above right] {$\scriptscriptstyle 2^{-{k_\ast}}$};

\end{tikzpicture}
\caption{\textit{Illustration of the proof of \eqref{eq:OS}.}} 
\label{fig:OSproof}
\end{figure}
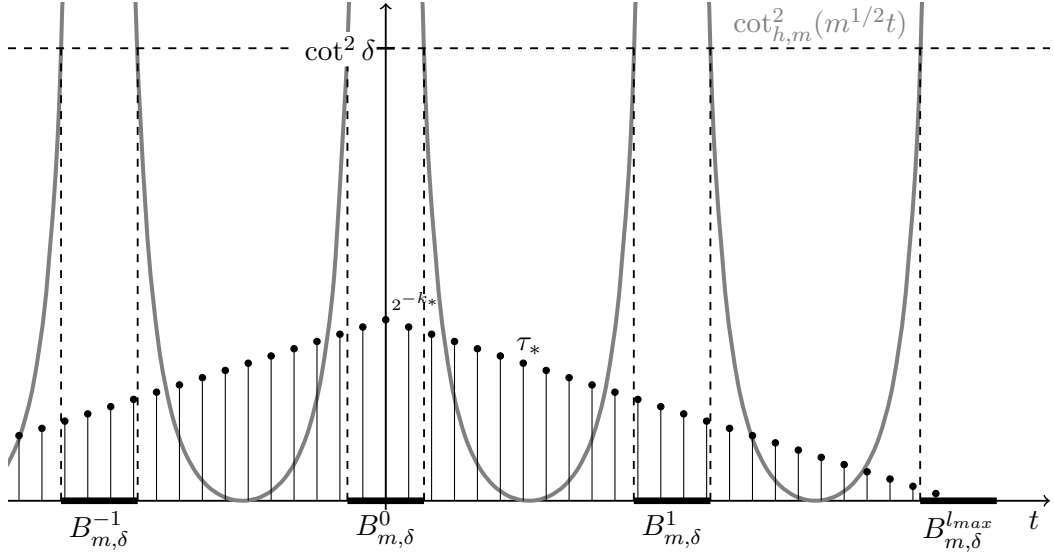

Let $\delta>0$.  To show \eqref{eq:OS}, we construct an explicit $B$.  Therefore, let $m\in\{m_1,m_2\}$ and define,
\[ B_{m,\delta}\ =\ \bigcup_{l\in\mathbb Z}B_{m,\delta}^l\quad\text{with}\quad B_{m,\delta}^l\ =\ \Bigr[\frac{\pi l-\delta}{\beta_{h^2m}m^{1/2}},\,\frac{\pi l+\delta}{\beta_{h^2m}m^{1/2}}\Bigr)\;. \]
Since for $t\in B_{m,\delta}^c$, $\cot_{h,m}^2(m^{1/2}t)\leq\cot^2\delta$,
\begin{equation}\label{eq:Bmdelta_1}
    \int_{B_{m,\delta}^c}\cot^2_{h,m}(m^{1/2}t)\,\tau_\ast(dt)\ \leq\ \cot^2\delta\,\tau_\ast(B_{m,c}^c)\ \leq\ \sin^{-2}\delta\quad\text{for $m\in\{m_1,m_2\}$.}
\end{equation}
Further, since $\operatorname{supp}(\tau_\ast)\subset\bigr[-h2^{k_\ast},\,h2^{k_\ast}\bigr]$,
\begin{equation}\label{eq:Bmdelta_sum}
    \tau_\ast(B_{m,\delta})\ =\ \sum_{l_{min}\leq l\leq l_{max}}\tau_\ast(B_{m,\delta}^l)
\end{equation}
with
\begin{align*}
    l_{min}\ &=\ \inf\bigr\{l\in\mathbb Z:\operatorname{supp}(\tau_\ast)\cap B_{m,\delta}^l\ne\emptyset\bigr\}\ \geq\ \inf\Bigr\{l\in\mathbb Z:-h2^{k_\ast}<\frac{\pi l+\delta}{\beta_{h^2m}m^{1/2}}\Bigr\}\\
    &=\ -\Bigr(\Bigr\lceil\frac{\beta_{h^2m}m^{1/2}h2^{k_\ast}+\delta}{\pi}\Bigr\rceil-1\Bigr)\ \geq\ -\Bigr\lceil\frac{\beta_{h^2m}m^{1/2}h2^{k_\ast}}{\pi}\Bigr\rceil \\
    l_{max}\ &=\ \sup\bigr\{l\in\mathbb Z:\operatorname{supp}(\tau_\ast)\cap B_{m,\delta}^l\ne\emptyset\bigr\}\ \leq\ \sup\Bigr\{l\in\mathbb Z:\frac{\pi l-\delta}{\beta_{h^2m}m^{1/2}}\leq h2^{k_\ast}\Bigr\}\\
    &=\ \Bigr\lfloor\frac{\beta_{h^2m}m^{1/2}h2^{k_\ast}+\delta}{\pi}\Bigr\rfloor\ \leq\ \Bigr\lceil\frac{\beta_{h^2m}m^{1/2}h2^{k_\ast}}{\pi}\Bigr\rceil
\end{align*}
where we used $\delta/\pi<1$ in the respective last steps.
The individual summands are bounded according to
\[ \tau_\ast(B_{m,\delta}^l)\ \leq\ \sup_{t\in\mathbb R}\tau_\ast(\{t\})\,\bigr|B_{m,\delta}^l\cap h\mathbb Z\bigr|\ \leq\ 2^{-k_\ast}\frac{2\delta}{\beta_{h^2m}m^{1/2}h} \]
where the absolute value denotes the counting measure.
Inserting into \eqref{eq:Bmdelta_sum} yields
\begin{align}\label{eq:Bmdelta_2}
    \tau_\ast(B_{m,\delta})\ &\leq\ \frac{2\delta}{\beta_{h^2m}m^{1/2}h2^{k_\ast}}(l_{max}-l_{min}+1)\nonumber\\
    &\leq\ \frac{2\delta}{\beta_{h^2m}m^{1/2}h2^{k_\ast}}\Bigr(2\Bigr\lceil\frac{\beta_{h^2m}m^{1/2}h2^{k_\ast}}{\pi}\Bigr\rceil+1\Bigr) \nonumber\\
    &\leq\ 2\max\Bigr(\frac4\pi,\,\frac6{m^{1/2}h2^{k_\ast}}\Bigr)\,\delta\nonumber\\
    &\leq\ 12\max\bigr(1,\,(m_1^{1/2}h2^{k_\ast})^{-1}\bigr)\,\delta\quad\text{for $m\in\{m_1,m_2\}$.}
\end{align}
Combining \eqref{eq:Bmdelta_1} and \eqref{eq:Bmdelta_2} shows for $B=B_{m_1,\delta}\cup B_{m_2,\delta}$,
\begin{align*}
    &\sup_{m\in\{m_1,m_2\}}\Bigr(\int_{B^c}\cot^2_{h,m}(m^{1/2}t)\,\tau_\ast(dt)\Bigr)^{1/2}|x-\tilde x|_{2S} + \tau_\ast(B) \\
    &\quad\leq\ \sup_{m\in\{m_1,m_2\}}\Bigr(\int_{B_{m,\delta}^c}\cot^2_{h,m}(m^{1/2}t)\,\tau_\ast(dt)\Bigr)^{1/2}|x-\tilde x|_{2S} + 2\sup_{m\in\{m_1,m_2\}}\tau_\ast(B_{m,\delta}) \\
    &\quad\leq\ \sin^{-1}\delta\,|x-\tilde x|_{2S} + 24\max\bigr(1,\,(m_1^{1/2}h2^{k_\ast})^{-1}\bigr)\,\delta \\
    &\quad\leq\ 2\delta^{-1}|x-\tilde x|_{2S}+24\max\bigr(1,\,(m_1^{1/2}t_\ast)^{-1}\bigr)\,\delta\quad\text{for $\delta\in(0,1)$.}
\end{align*}
\end{proof}

\subsection{Proof of Theorem~\ref{thm:NUTSmixing}}\label{sec:mixing_proof}

\begin{proof}[Proof of \Cref{thm:NUTSmixing}]

By \Cref{lem:MultHMCtoUnifHMC}, 
\[ \pi_{\NUTS}(x,\cdot)\,\ind_{\{v\in E^{2S}_{\alpha,r}\}\cap A_I(x,v)}\ =\ \pi_{\mathrm{HMC}(\tau_\ast)}(x,\cdot)\,\ind_{\{v\in E^{2S}_{\alpha,r}\}\cap A_I(x,v)}\quad\text{for all $x\in D^{2S}_{\alpha}$,} \]
i.e., NUTS fits the definition of an accept/reject chain with accept chain $\pi_{\mathrm{HMC}(\tau_\ast)}$ and accept event $\{v\in E^{2S}_{\alpha,r}\}\cap A_I(x,v)$, cf. \eqref{eq:ar}.

To obtain mixing time bounds for NUTS, we apply the coupling framework for accept/reject chains given in \Cref{thm:ARmix} with state space $S=\mathbb R^d$, metric induced by $|\cdot|_{2S}$, and domain $D=D^{2S}_\alpha$ for a suitable $\alpha=(\alpha_1,\alpha_2)$ to which we localize.
Therefore, we verify Assumptions (i)-(iv).

For HMC with integration time distribution $\tau_\ast$, Wasserstein contraction (Assumption (i)) holds by \Cref{lem:contr} with rate $\rho=\Omega(\min(m_1^{1/2}t_\ast,\,1)^2)$.
Partial total variation to Wasserstein regularization (Assumption (ii)) holds by \Cref{lem:OS} with $C_{\mathrm{Reg}}=O(\min(m_1^{1/2}t_\ast,\,1)^{-1})$ and an absolute constant $0<c<1$, which results from inserting $\delta=\Theta(\min(m_1^{1/2}t_\ast,\,1))$ into \eqref{eq:OS}.

Assumption (iii) requires $\mathfrak E\in\mathbb N$ and $b>0$ to be chosen such that
\begin{align}\label{eq:requiredBd}
    &2\,\mathfrak E\,\sup\nolimits_{x\in D^{2S}_\alpha}\mathbb{P}\bigr(\bigr(\{v\in E^{2S}_{\alpha,r}\}\cap A_I(x,v)\bigr)^c\bigr)\ +\ C_{\mathrm{Reg}}\diam_{|\cdot|_{2S}}(D^{2S}_\alpha)\,e^{-\rho(\mathfrak E-1)}\ +\ b\nonumber\\
    &\quad\leq\ 1\ -\ c\;.
\end{align}
Inserting $\rho,C_{\mathrm{Reg}},c$ and the bound
\[ \sup\nolimits_{x\in D^{2S}_\alpha}\mathbb P\bigr(\bigr(\{v\in E^{2S}_{\alpha,r}\}\cap A_I(x,v)\bigr)^c\bigr)\ \leq\ 
\begin{aligned}[t]
&8\,e^{-\frac18\min_{i\in\{1,2\}}r_i^2/d_i}\\
&+2\hbar^2\max_{i\in\{1,2\}}\bigr(m_i\max(\alpha_i,r_i)+\hbar^2m_i^2d_i\bigr)\;,
\end{aligned} \]
valid with $\hbar=0$ for Hamiltonian flow and $\hbar=h$ for leapfrog flow as shown in \Cref{lem:MultHMCtoUnifHMC}, together with the fact that $\diam_{|\cdot|_{2S}}(D^{2S}_\alpha)=O(d^{1/2})$, we conclude the existence of an absolute constant $b>0$ such that \eqref{eq:requiredBd} holds with
\begin{align}
\mathfrak E\ &=\ \widetilde\Theta(\rho^{-1})\ =\ \widetilde\Theta\bigr(\min(m_1^{1/2}t_\ast,\,1)^{-2}\bigr)\;, \nonumber\\
r_i\ &=\ \widetilde\Theta(d_i^{1/2})\quad\text{for $i\in\{1,2\}$, and, in case of leapfrog flow, all} \label{eq:proof_r}\\
h\ &\leq\ \bar h\ =\ \widetilde\Omega\bigr(\mathfrak E^{-1/2}m_2^{-1/2}\max(\alpha_1,\alpha_2,r_1,r_2,d_1^{1/2},d_2^{1/2})^{-1/2}\bigr)\;.\label{eq:proof_h}
\end{align}

This results in a horizon $\mathfrak H=\widetilde\Theta(\mathfrak E)=\widetilde\Theta\bigr(\min(m_1^{1/2}t_\ast,\,1)^{-2}\bigr)$.
Assumption (iv) requires exit probability bounds from $D^{2S}_\alpha$ for NUTS over this horizon initialized in both $x^0$ and the stationary distribution $\gamma^{2S}$.
Therefore, we choose suitably large $\alpha_1$ and $\alpha_2$.
For start in $x^0$, \Cref{lem:stab_leapfrog} asserts
\begin{equation}\label{eq:stab_NUTS_pf}
    \mathbb P(T^{2S}_{\alpha_\hbar(\mathfrak H)}\leq \mathfrak H)\ \leq\ 8\mathfrak H\,e^{-\frac18\min_{i\in\{1,2\}}r_i^2/d_i}
\end{equation}
with
\begin{equation}\label{eq:alphan_pf}
    \alpha_\hbar(\mathfrak H)\ =\ \bigr(\max(\alpha^0_i,r_i)+\mathfrak H(r_i+\hbar^2m_id_i)\bigr)_{i\in\{1,2\}}
\end{equation}
for $\hbar=0$ in case of Hamiltonian flow and $\hbar=h$ in case of leapfrog flow.
Choosing $\alpha=\alpha_\hbar(\mathfrak H)$ and for $r_i$ as in \eqref{eq:proof_r}, \eqref{eq:stab_NUTS_pf} provides the required exit probability bound for start in $x^0$.
For start in stationarity $\bar x^0\sim\gamma^{2S}$, note that $\mathbb P(\bar x^0\in D^{2S}_{\bar\alpha^0})\geq1-\eps/8$ for $\bar\alpha^0_i=\widetilde\Theta(d_1^{1/2})$, $i=1,2$, by \eqref{eq:D2Sbound}.
Hence, the exit probability bound in this case results from restricting to start in $D^{2S}_{\bar\alpha^0}$ and applying \eqref{eq:stab_NUTS_pf} with the corresponding $\alpha_\hbar(\mathfrak H)$.

Inserting our choices into \eqref{eq:proof_h} yields
\[ \bar h\ =\ \widetilde\Omega\bigr(m_2^{-1/2}d^{-1/4}\min(m_1^{1/2}t_\ast,\,1)^2\bigr)\;. \]

Thus, all assumptions of \Cref{thm:ARmix} are satisfied, thereby establishing
\[ \tmix(\eps,x^0)\ \leq\ \mathfrak H\ =\ \widetilde\Theta\bigr(\min(m_1^{1/2}t_\ast,\,1)^{-2}\bigr)\;. \]

The dichotomy follows immediately from the mixing time bound and the dichotomy shown in \Cref{prop:2S}.
\end{proof}

\printbibliography

\end{document}